\documentclass[11pt]{amsart}
\usepackage{esint, amssymb}
\usepackage{stmaryrd}
\usepackage{fancyhdr}
\usepackage{enumerate}
\usepackage{mathtools}
\usepackage{tikz}
\usetikzlibrary{automata,positioning}
\usepackage{bbm}
\usepackage{blkarray}
\usepackage{xfrac}
\usepackage{bigints}

\usepackage[text={420pt,650pt},centering]{geometry}

\usepackage{xcolor}
\usepackage{graphicx}
\usepackage{MnSymbol}
\usepackage[colorlinks=true, pdfstartview=FitV, linkcolor=blue, citecolor=blue, urlcolor=blue,pagebackref=false]{hyperref}
\usepackage{microtype}

\usepackage{bm}
\usepackage{scalerel} 
\usepackage{dsfont}
\usepackage[font={footnotesize}]{caption}

\usepackage{verbatim}

\usepackage{scalerel}[2014/03/10]
\usepackage[usestackEOL]{stackengine}
\def\dashint{\,\ThisStyle{\ensurestackMath{%
  \stackinset{c}{.2\LMpt}{c}{.5\LMpt}{\SavedStyle-}{\SavedStyle\phantom{\int}}}%
  \setbox0=\hbox{$\SavedStyle\int\,$}\kern-\wd0}\int}
  
\newcommand{\dashsum}{%
    \mathop{
        \mathchoice
        {\ooalign{$\displaystyle\sum$\cr\vphantom{$\displaystyle\sum$}\hidewidth$\displaystyle\kern 0.2ex\vcenter{\hrule height 0.8pt width 0.9em}\kern 0.2ex$\hidewidth}}
        {\ooalign{$\textstyle\sum$\cr\vphantom{$\textstyle\sum$}\hidewidth$\textstyle\kern 0.2ex\vcenter{\hrule height 0.6pt width 0.75em}\kern 0.2ex$\hidewidth}}
        {\ooalign{$\scriptstyle\sum$\cr\vphantom{$\scriptstyle\sum$}\hidewidth$\scriptstyle\kern 0.15ex\vcenter{\hrule height 0.4pt width 0.6em}\kern 0.15ex$\hidewidth}}
        {\ooalign{$\scriptscriptstyle\sum$\cr\vphantom{$\scriptscriptstyle\sum$}\hidewidth$\scriptscriptstyle\kern 0.15ex\vcenter{\hrule height 0.3pt width 0.5em}\kern 0.15ex$\hidewidth}}
    }
}

\fancypagestyle{first}{%
  \fancyhf{}
  \headrulewidth{0pt}

}

\newcommand{\A}{\mathbb A}
\newcommand{\N}{\mathbb N}
\newcommand{\Z}{\mathbb Z}

\newcommand{\R}{\mathbb R}

\renewcommand{\H}{\mathbb H}

\renewcommand{\P}{\mathbb P}

\newcommand{\T}{\mathbb T}
\newcommand{\Rd}{\mathbb R^d}
\newcommand{\Zd}{\mathbb{Z}^d}
\newcommand{\Nd}{\mathbb{N}^d}
\newcommand{\Td}{\mathbb T^d}
\newcommand{\1}{\mathbbm 1}
\newcommand{\uL}{\underline{L}}
\newcommand{\uH}{\underline{H}}

\newcommand{\al}{\alpha}

\newcommand{\ep}{\varepsilon}

\renewcommand{\a}{\mathbf{a}}

\newcommand{\g}{\mathbf{g}}
\newcommand{\h}{\mathbf{h}}

\newcommand{\ahom}{{\overbracket[1pt][-1pt]{\mathbf{a}}}}
\newcommand{\Phom}{{\overline{P}}}
\newcommand{\Qhom}{{\overline{Q}}}

\newcommand{\mc}[1]{\mathcal{#1}}

\renewcommand{\hat}[1]{\widehat{#1}}

\definecolor{officegreen}{rgb}{0.0, 0.5, 0.0}

\newcommand{\hyp}{{\mathrm{hyp}}}
\newcommand{\kin}{{\mathrm{kin}}}

\DeclarePairedDelimiter{\abs}{\lvert}{\rvert}
\DeclarePairedDelimiter{\inp}{(}{)}

\DeclarePairedDelimiter{\insb}{\{}{\}}
\DeclarePairedDelimiter{\floor}{\lfloor}{\rfloor}
\DeclarePairedDelimiter{\ceil}{\lceil}{\rceil}
\DeclarePairedDelimiter{\inprod}{\langle}{\rangle}
\DeclarePairedDelimiter{\norm}{\|}{\|}
\DeclarePairedDelimiter{\snorm}{\llbracket}{\rrbracket}

\newtheorem{thm}{Theorem}
\numberwithin{thm}{section}
\newtheorem{lem}{Lemma}
\numberwithin{lem}{section}

\numberwithin{cor}{section}
\newtheorem{prop}{Proposition}
\numberwithin{prop}{section}
\numberwithin{equation}{section}

\newcommand{\tref}[1]{Theorem~\ref{t.#1}}
\newcommand{\pref}[1]{Proposition~\ref{p.#1}}
\newcommand{\lref}[1]{Lemma~\ref{l.#1}}
\newcommand{\cref}[1]{Corollary~\ref{c.#1}}

\newcommand{\eref}[1]{(\ref{e.#1})}

\theoremstyle{definition}

\numberwithin{defn}{section}

\numberwithin{ex}{section}

\theoremstyle{remark}

\numberwithin{rem}{section}

\hypersetup{
    colorlinks,
    citecolor=black,
    filecolor=black,
    linkcolor=black,
    urlcolor=black
}

\begin{document}
\title[Large-Scale Regularity for the Periodic Kinetic Fokker-Planck equation]{Large-Scale Regularity for the Periodic Kinetic Fokker-Planck equation}

\date{\today} 

\begin{abstract}
We first prove a homogenization result for the fundamental solution of the linear kinetic Fokker Planck equation. We show that this solution converges, in an averaged $L^2$ sense, to the fundamental solution of an effective heat equation with constant effective diffusivity determined by corrector functions solving associated cell problems on the torus. A key feature of the proof is the necessity of second-order correctors to control the averaging of the velocity variable, and the handling of a non-divergence form error term arising from limited spatial regularity of solutions.

Additionally, building on this homogenization result, we establish a large-scale regularity result for solutions of this Fokker Planck equation. More specifically, we show that solutions by heterogeneous polynomials, analogous to Taylor polynomials,  with an explicit error on large scale domains. Furthermore, we show that in a larger regime, this approximating polynomial solves this Fokker Planck equation.
\end{abstract}

\author[P. Gaddy]{Philip Gaddy}
\address[P. Gaddy]{Courant Institute of Mathematical Sciences, New York University, 251 Mercer St., New York, NY 10012}
\email{philip.gaddy@courant.nyu.edu}

\maketitle

\section{Introduction}
\subsection{Motivation and Summary of Results}
We are interested in studying the large scale behavior of solutions to the forced kinetic Fokker-Planck equation, given by:
\begin{equation}
\label{e.hypoeq}
\partial_t f-\nabla_v\cdot\a\nabla_v f + v\cdot\a\nabla_v f -v\cdot \nabla_x f + \nabla H \cdot \nabla_v f = 0\,,
\end{equation}
with some smooth initial condition. For our study, $\a(x,v):\Rd \times \Rd \to \R^{d\times d}$ is a bounded, positive definite, symmetric matrix, and $H: \Rd \to \R$ is a bounded and $\Zd$-periodic function. Of particular interest is the operator defined on the left hand side of the above equation, given by 
$$ Lf :=  -\nabla_v\cdot\a\nabla_v f + v\cdot\a\nabla_v f -v\cdot \nabla_x f + \nabla H \cdot \nabla_v f \,.$$
This operator is the generator for a special case of a Langevin diffusion process, $(X_t, V_t)$, which can be defined by the SDE
\begin{equation}
\label{e.hypomarkoveq}
\begin{cases}
dX_t = V_tdt \\
dV_t \,\,= (-\a V_t + \nabla H(X_t))dt + \sqrt{2}\a^{\sfrac{1}{2}}dW_t\,,
\end{cases}
\end{equation}
where $W_t$ is a standard Brownian motion.  In a more physical sense, this equation arises by considering a particle moving subject to the force field $\nabla H$ together with friction and noise, given by the $\a$ matrix.

This equation also proves to be of interest in the study of degenerate elliptic equations. H\"ormander in \cite{H} and Kohn and Nirenberg in \cite{KN1, KN2} originally consider equations in the form of \eref{hypoeq} in this context to study the generalizability of the regularization properties of elliptic equations. In \cite{AAMN}, Albritton, Armstrong, Mourrat, and Novack develop a variational framework for the kinetic Fokker-Planck equation that parallels the classical energy methods for elliptic equations. In addition to a well-posedness theory for weak solutions of \eref{hypoeq}, they establish Poincar\'e inequalities, Caccioppoli estimates, and fractional Sobolev embeddings in the hypoelliptic Sobolev space $H^1_\hyp$, providing the functional analytic infrastructure that we heavily rely on throughout the rest of this work. Zhu \cite{Z} uses renormalization methods in the tradition of DiPerna-Lions to prove existence and uniqueness for solutions to a more general version of equation \eref{hypoeq} in bounded domains, as well as to provide H\"older regularity estimates. Silvestre \cite{S} independently establishes local H\"older continuity estimates up to the spatial boundary using De Giorgi methods, obtaining precise asymptotic estimates near the influx boundary.

This problem has also been considered probabilistically. Zhang \cite{FS} employed the tools of Malliavin calculus to establish the existence of smooth densities, which serve as fundamental solutions, for kinetic Fokker-Planck operators exhibiting anisotropic nonlocal dissipativity. Leli\`evre, Ramil, and Reygner \cite{Lelievre} studied the absorbed Langevin process in cylindrical domains, proving the existence and regularity of its transition density and deriving an explicit Gaussian upper bound.

Further investigations have explored the effects of boundaries and geometric constraints in kinetic Fokker-Planck equations. Existence and regularity results in this context have been obtained by Carrillo \cite{Car} for Vlasov-Poisson-Fokker-Planck systems, by Nier \cite{Nierbook} in a geometric setting, and more recently by Silvestre \cite{S}, Zhu \cite{Z}, Bernard \cite{Ber}, among others. Finally, the long-time behavior of kinetic Fokker-Planck dynamics, often referred to as hypocoercivity, has been analyzed through both probabilistic and analytic techniques. Probabilistic approaches have been developed by Rey-Bellet and Thomas \cite{RBT}, while analytic treatments can be found in the works of Herau \cite{H1}, Villani \cite{V}, Mischler and Mouhot \cite{Misc}, Baudoin \cite{baudoin}, and others, including the contributions of Hwang et al. \cite{Hwang}.

Previous study of the homogenization of this equation focuses heavily on the Markov process defined in \eref{hypomarkoveq} and, as such, considers it from a more probabilistic perspective. Rodenhausen \cite{R} proves a related homogenization result from the perspective of Langevin dynamics, establishing in particular the validity of the Einstein relation between the effective diffusion constant and mobility. Hairer and Pavlotis in \cite{HP1, HP2} show a more standard homogenization result using techniques specific to the study of stochastic processes. Similarly Papanicolaou and Varadhan in \cite{PV} uses similar techniques to consider the case where $H$ is a random field rather than periodic. In contrast, in this paper we use the functional analysis results developed in \cite{AAMN} to consider this problem from an entirely analytic perspective. Namely, we show that the classical energy methods used to prove homogenization results for elliptic equations may be adapted to apply in this setting. In order to adapt these results, we had to address two major obstacles. Firstly, the degenerate structure in the $x$ variable. Solving this requires not only the use of the H\"ormander inequality instead of more standard elliptic estimates, but also the use of a nonstandard form of the error. This second point is addressed in more detail in Section 4. Additionally, we also have to address the averaging out rather than homogenization of the velocity variable. This requires the use of the second order correctors, in addition to the first order correctors typically present in homogenization proofs.

In addition to the analytic proof of a homogenization result discussed above, we also provide a proof of large-scale regularity for the equation \eref{hypoeq}. Originally discussed by Avellaneda and Lin in \cite{AL1, AL2}, large-scale regularity is one of the most crucial facets of homogenization theory. Avellaneda and Lin develop compactness methods and qualitative large-scale regularity theory in H\"older and Sobolev spaces in \cite{AL1}, and derive Liouville-type theorems as a consequence in \cite{AL2}. Building on the quantitative homogenization theory developed in the elliptic setting, Armstrong, Kuusi, and Smart \cite{AKS} establish large-scale regularity for periodic elliptic equations, which this present paper adapts to the hypoelliptic setting. The methods developed in \cite{AKS} prove to be very general. In particular, and using our homogenization result proved in the first half of this paper, we then go on to adapt the proof for large-scale analyticity to provide a new large-scale regularity result for solutions of \eref{hypoeq}.

\subsection{Explanation of Main Results}
We first begin with a statement of the main homogenization result we will be proving. Before stating this result, we first define notation we use to state the result. For $T>1$, we let 
$$\norm{f}_{\uL^2(\mc{Q}_T; L^2_\gamma)} = \dashint_{\frac{T}{2}}^T\int_{B_{\sqrt{T}}(0)\times \Rd} f^2 (2\pi)^{-d}e^{-H(x) - \frac{\abs{v}^2}{2}}\,dxdvdt\,.$$
Additionally, we say that $f$ solves the homogenized equation provided
\begin{equation}
\label{e.hypohomeq}
\partial_t f - \nabla_x\cdot \ahom\nabla_x f = 0\,,
\end{equation}
where $\ahom_{ij} = -\inprod{v_j\phi_i}$, where $\phi_{i}$ are the correctors associated with \eref{hypoeq}. These are more clearly defined in \lref{correctorexist}.

\begin{thm}
\label{t.perhypohom}
Let $P(t,x,v,y,w)$ and $\Phom$ denote the Green functions for the hypoelliptic equation \eref{hypoeq} and the homogenized equation \eref{hypohomeq}, respectively. Then there exists $C< \infty$ such that, for every $t_0,t \geq 1$, with $t_0\leq \frac{t}{2}$, and $y,w\in \Rd$, 
$$\norm*{P(\cdot,y,\cdot,w) - \Phom(\cdot,\cdot-y)}_{\uL^2(\mc{Q}_t; L^2_\gamma)} \leq C\inp*{t^{-\frac{1}{2}} + \inp*{\frac{t_0}{t}}^\frac{d}{4} + t_0^{-\frac{d}{4}}\inp*{\frac{t_0}{t}} + t_0^{-\frac{1}{2}} \log\inp*{\frac{t}{t_0}}}t_0^{-\frac{d}{4}}\,.$$
\end{thm}

To prove this estimate, we first begin by picking a mesoscopic time where we begin our analysis. From here, we define $\Qhom$ to be the solution of the homogenized equation that equals $P$ at this mesoscopic time. It is then relatively easy to show that $\Phom$ and $\Qhom$ are close in the metric defined in the theorem as they solve the same equation and their initial conditions have the same mean.

The main homogenization proof appears in trying to then estimate the difference between $P$ and $\Qhom$. Here, we look at the two scale expansion and construct an approximate solution of \eref{hypoeq} by attaching the correctors to $\Qhom$. This approximate solution is then necessarily close to $\Qhom$ by construction. The majority of our proof lies in showing that this approximate solution is close to $P$. For this, we need not only the first order correctors typically needed in proofs of this type, but also the second order correctors to better control the fluctuations in the velocity variable as it drops out over the course of this proof. Another nonstandard difficulty is that the error coming from this approximate solution is in non-divergence form. As we will see in the discussion leading up to the proof, there actually is a divergence form for this error. However, using this form to bound our solutions later requires having much sharper estimates on the spatial regularity of solutions to \eref{hypoeq} than we generically do.

In proving this, we first construct the correctors and approximate solution as usual. One slight difference to note is that our approximate solution includes a second order corrector, which we call $\psi$. This extra corrector proves necessary to help control the convergence of a solution to scaled equation to it's $v$ average. which we then show is close to the homogenized equation. Once this has been defined, we then prove several estimates on the approximate solution. This, combined with the Nash-Aronson type bound proved in \cite{Nash}, are then enough to finish the proof of this theorem.

Then, from this, we prove the following regularity result. First, we let 
$$Q_r := \inp*{-\frac{1}{2}r, \frac{1}{2} r}^d\,,$$
and $\A_m$ and $\A_m^0$ be the spaces of heterogeneous polynomials, which we define in section 3.
\begin{thm}[Large-Scale Regularity]
\label{t.largescale}
There exists $C < \infty$ such that, for every $m\in\N$ with $m\geq 0$, $R\geq 2Cm$ and $f\in H_\hyp^1(Q_R)$ a solution of 
$$-\nabla_v\cdot\a\nabla_v f +v\cdot\a\nabla_v f -v\cdot\nabla_x f + \nabla H\cdot \nabla_v f= 0 \quad \text{ in } Q_R\,,$$
there exists $\psi \in\A_m$ such that, for every $r\in [Cm, R]$, 
$$\norm{f-\psi}_{\uL^2(Q_r;L^2_\gamma)} \leq \inp*{\frac{Cr}{R}}^{m+1} \norm{f}_{\uL^2(Q_R; L^2_\gamma)}\,.$$
Furthermore, if $R\geq 2Cm^2$, then there exists $\psi \in \A_m^0$ satisfying the same bound for all $r\in [Cm^2, R]$.
\end{thm}

The variational structure presented in \cite{AAMN} allows us to adapt the proof of the large scale regularity presented in \cite{AKS} to the hypoelliptic case. It is worth noting, however, that the statement here differs from the principal large-scale regularity result presented in the aforementioned paper. Specifically, we show that $f$ can only be approximated to this degree of accuracy by a heterogeneous polynomial that is a to solution of the equation when the scale $R$ is well above the minimum scale, $Cm$. When $R$ approaches $Cm$, we can merely require that $\psi\in \A_m$ and not that $\psi$ also solve the equation. Removing this qualification may be possible, but will require a much more fine tuned approach. This is the result of the proof of \cite{AKS}[Lemma 3.4] requiring an additional condition; the condition $r \geq Cm$ given there is not sufficient to control the sums on the right-hand side of the key inequality. This then requires imposing the stronger condition $R\geq Cm^2$ in the secondary statement of \tref{largescale}.

\section{Functional Preliminaries and Inequalities}
Throughout the course of this paper, we heavily rely on the functional analysis framework for solutions of \eref{hypoeq} as well as the time independent developed throughout \cite{AAMN}. As most of the function spaces and notation used are non standard, we use this section to cover these definitions as well as several results that will be important for proving \tref{perhypohom} and \tref{largescale} in sections 4 and 5 respectively. This first subsection discusses the time dependent equation \eref{hypoeq}, while the following discusses the time independent case.

\subsection{Time Dependent Results}

We begin by defining the following measures
\begin{equation*}
\left\{  
\begin{aligned}
& d\sigma(x) := \exp\left( -H(x) \right) \,dx\,, \\
& d\gamma(v):= \frac1{(2\pi)^d} \exp\left( -\frac12|v|^2 \right)\,dv\,, \\
& dm(x,v) := d\sigma(x) d\gamma(v)\, , \\
& dm_t(t,x,v) := dtd\sigma(x) d\gamma(v)\, .
\end{aligned}
\right.
\end{equation*}
First, we notice that, since $H$ is bounded, in particular $\abs{H}\leq \Lambda$, there are positive constants $C(\Lambda), c(\Lambda) <\infty$ such that
$$c(\Lambda) \int_U \abs{f} \,dx \leq \int_U \abs{f}\,d\sigma \leq C(\Lambda) \int_U \abs{f}\,dx\,,$$
for any $f\in L^1(U)$. As such, the measure $\sigma$ is equivalent to the Lebesgue measure, and we will swap between the two frequently throughout some proofs. This measure is important largely for simplifying several integration by parts calculations later on.

For $I\subset\R_+$ and $U\subset \Rd$, we define the space $H^1_\kin(I\times U)$ is defined by
\begin{equation*}  
H^1_\kin(I\times U) := \insb*{ f \in L^2\left(I\times U;H^1_\gamma\right) \ : \ \partial_t f - v \cdot \nabla_x f \in L^2(I \times U;H^{-1}_\gamma)}\,,
\end{equation*}
and equip it with the norm
\begin{equation*}
\left\| f \right\|_{H^1_\kin(I\times U)} 
:=
\left( \left\| f \right\|_{L^2(I\times U;H^1_\gamma)}^2
+ 
\left\| \partial_tf - v \cdot \nabla_x f \right\|_{L^2(I\times U;H^{-1}_\gamma)}^2 \right)^{\frac12}\,.
\end{equation*}
We also define the seminorm
\begin{equation*}
\left\llbracket f \right\rrbracket_{H^1_\kin(I\times U)} 
:=
\left( \left\| \nabla_v f \right\|_{L^2(I\times U;L^2_\gamma)}^2
+ 
\left\| \partial_tf - v \cdot \nabla_x f \right\|_{L^2(I\times U;H^{-1}_\gamma)}^2 \right)^{\frac12}
\end{equation*}
so that we have $\left\| f \right\|_{H^1_\kin(I\times U)}^2 = \left\llbracket f \right\rrbracket_{H^1_\kin(I\times U)}^2 + \left\| f \right\|_{L^2(I\times U;L^2_\gamma)}^2$. For most of this paper, we will work with $I = (t-1,t)$ for some $t\in \R$. To simplify notation, we set $U_t := (t-1,t)\times U$.

We say that $f \in H^1_\kin(I\times U)$ is a \emph{weak solution} of the equation
\begin{equation*}
\partial_t f -\nabla_v\cdot\a\nabla_v f + v \cdot \a\nabla_v f - v \cdot \nabla_x f + \nabla H(x) \cdot \nabla_v f = 0 \quad \mbox{in} \ U \times \Rd
\end{equation*}
provided that, for all $g\in L^2(I\times U; H^1_\gamma)$
\begin{equation}
\label{e.tweaksoldef1}
\int_{I\times U\times \Rd} 
\nabla_vg\cdot \a\nabla_v f \,dm_t
=
\int_{I\times U\times \Rd} 
g\left(
v\cdot \nabla_xf
-\partial_t f
- \nabla H\cdot \nabla_vf \right)\,dm_t\,.
\end{equation}

To reduce later notation, we define a linear operator $(\partial_t + B):H^1_\kin(I\times U)\to L^2(I\times U;H^{-1}_\gamma)$ by
\begin{equation*}
Bf := -v\cdot \nabla_x f + \nabla H \cdot \nabla_vf\, . 
\end{equation*}
Additionally, we define $\nabla_v^*$ as the formal adjoint of $\nabla_v$ with respect to the inner product of~$L^2_\gamma$. To be more precise, $\nabla_v^*$ is a linear map from $(L^2_\gamma)^d$ to $H^{-1}_\gamma$ given by
\begin{equation*}
\int_{\Rd} f \nabla_v^* \g \,d\gamma
=
\int_{\Rd} \nabla_v f\cdot \g \,d\gamma\,. 
\end{equation*}
With these two definitions, we can rewrite \eref{tweaksoldef1}  as follows
\begin{equation}
\label{e.tweaksoldef2}
\forall g\in L^2(U;H^1_\gamma),  \quad
\int_{U\times \Rd} 
\left( \nabla_vg\cdot \a\nabla_v f + g (\partial_t+B)f \right) \,dm
=
0\,.
\end{equation}
Consequently, we can rewrite \eref{hypoeq} as 
$$\partial_tf + \nabla_v^* \a\nabla_v f + Bf = 0\,.$$

Now that we have defined the major function space we will be working with, we state a few inequalities satisfied by elements of this space. We begin by proving a Caccioppoli inequality for solutions of \eref{hypoeq}
Let 
$$T_r = (-r^2, 0]\times B_{r}\,.$$
Then, 
\begin{lem}[Caccioppoli inequality]
\label{l.tcaccioppoli}
There exists a constant~$C(d,\Lambda)<\infty$ such that, for every $r>0$ and $f \in H^1_\kin(T_{2r})$ satisfying
\begin{equation}
\label{e.tcaccio.pde}
\partial_t f -\nabla_v\cdot\a\nabla_v f + v\cdot \a\nabla_v f - v\cdot \nabla_x f + \nabla H\cdot \nabla_v f = 0 \quad \mbox{in} \ T_{2r} \times\Rd\,, 
\end{equation}
we have that
\begin{equation*}
\left\llbracket f \right\rrbracket_{H^1_{\kin}(T_{r}) }
\leq 
C r^{-1}
\left\| f \right\|_{L^2(T_{2r};L^2_\gamma)}\,.
\end{equation*}
\end{lem}
\begin{proof}
We begin by selecting $\phi\in C^\infty_c(T_{2r})$ satisfying
$$
\1_{T_r}\leq \phi \leq \1_{T_{2r}}
\quad 
\text{and}
\quad
\abs{\partial_t \phi} + \abs{\nabla\phi}^2 \leq Cr^{-2}\, .
$$
Testing \eqref{e.tweaksoldef2} of the equation with $\phi^2 f$ and using the ellipticity of $\a$, we see that
\begin{equation}
\label{e.tcaccio.plugin}
\int_{T_{2r}\times\Rd} \phi^2 \left| \nabla_vf \right|^2 \,dm_t
\leq 
- \int_{T_{2r}\times\Rd} \phi^2 f (\partial_t + B)f \,dm_t\,. 
\end{equation} 
Integrating by parts, the right hand side becomes
\begin{align*}
-\int_{T_{2r}\times\Rd} \phi^2 f Bf \,dm_t &= -\int_{T_{2r}\times\Rd} 2\phi f \nabla\phi\cdot \nabla_v f \,dm_t \\
&\leq \frac12\int_{T_{2r}\times\Rd} \phi^2 \left| \nabla_vf \right|^2 \,dm_t + 2\int_{T_{2r}\times\Rd} \left|\nabla\phi \right|^2 f^2 \,dm_t\,,
\end{align*}
where the last step follows by applying Cauchy's inequality. Similarly, we can integrate by parts again to see that
\begin{align*}
-\int_{T_{2r}\times \Rd} \phi^2 f\partial_t f \,dm_t &= -\frac{1}{2}\int_{T_{2r}\times \Rd} \phi^2 \partial_t f^2 \,dm_t \\
&= \int_{T_{2r}\times \Rd} \phi f^2 \partial_t \phi \,dm_t\,.
\end{align*}
Combining this with \eref{tcaccio.plugin} and using the properties of $\phi$, 
\begin{equation}
\label{e.tcaccgradbnd}
\int_{T_r\times\Rd} \abs{ \nabla_vf}^2 \,dm_t
\leq 
Cr^{-2} \int_{T_{2r} \times\Rd} 
f^2\,dm_t\,. 
\end{equation}
This yields the desired estimate of $\left\| \nabla_v f \right\|_{L^2( T_r; L^2_\gamma)}$. To estimate $\left\| \partial_tf -v\cdot \nabla_x f \right\|_{L^2(T_r; H^{-1}_\gamma)}$, we test \eqref{e.tweaksoldef2} with $g \exp(H)$, for $g\in L^2(T_r;H^1_\gamma)$, and use \eqref{e.tcaccgradbnd} and the upper bound on $\a$ to see that
\begin{align*}
\left| \int_{T_r\times\Rd} g (\partial_t f - v\cdot \nabla_x f) 
\,dtdx\,d\gamma \right| 
& =
\left| \int_{T_r\times\Rd} \left( \a\nabla_vg +g \nabla H \right) \cdot \nabla_vf
\,dtdx\,d\gamma \right| 
\\ & 
\leq C \left\| g \right\|_{L^2(T_r;H^1_\gamma)} 
\left\| \nabla_v f \right\|_{L^2(T_r;L^2_\gamma)}
\\ & 
\leq 
Cr^{-1} \left\| g \right\|_{L^2(T_r;H^1_\gamma)} 
\left\| f \right\|_{L^2(T_{2r};L^2_\gamma)}\,.
\end{align*}
Taking the supremum over $g\in L^2(T_r;H^1_\gamma)$ with $\left\| g \right\|_{L^2(T_r;H^1_\gamma)}\leq 1$ yields 
\begin{equation*}
\left\| \partial_t f - v\cdot \nabla_x f \right\|_{L^2(T_r;H^{-1}_\gamma)} \leq Cr^{-1}\left\| f \right\|_{L^2(T_{2r};L^2_\gamma)}\,. 
\end{equation*}
Combining \eqref{e.tcaccgradbnd} and the previous line yields
\begin{align*}
\snorm{f}_{H^1_{\kin}(T_r) }^2
&=
\left\| \nabla_v f \right\|_{L^2(T_r; L^2_\gamma)}^2
+
\left\| \partial_t f - v\cdot \nabla_x f \right\|_{L^2(T_r;H^{-1}_\gamma)}^2\\
&\leq  Cr^{-2}\left\| f \right\|^2_{L^2(T_{2r};L^2_\gamma)}\,,
\end{align*}
completing the proof.
\end{proof}

Additionally, we will use the following bound on the space regularity of functions in $H^1_{\kin}$. This exact bound is stated in a remark to Proposition $6.6$ in \cite{AAMN}, but we state this as the main proposition as it is the bound we make use of 
\begin{prop}[H\"ormander Inequality of $H^1_\kin$]
\label{p.thormander}
There exists $C(d)<\infty$ such that, for $f\in H^1_\kin((\frac{T}{2}, T)\times \Rd)$ and $\alpha\in [0, \frac{1}{3})$, 
$$\norm{f}_{L^2((\frac{T}{2}, T); H^\alpha(\Rd; H^{-1}_\gamma))} \leq C\norm{f}_{H^1_\kin((\frac{T}{2}, T)\times \Rd)}\,.$$
\end{prop}

Next, we provide a bound which we will use in the proof of \tref{perhypohom}. We essentially show that, with respect to the $L^2_\gamma$ norm, we can relate multiplication by $v$ to taking a derivative in the $v$ variable.This is proven in \cite{AAMN}[(3.34)], but we reproduce this discussion here as a lemma for the sake of readability.
\begin{lem}
\label{l.orlicz}
There exists $C(d)<\infty$ such that, for $f\in L^2(U; H^1_\gamma)$, 
$$\norm{v f}_{L^2(U; L^2_\gamma)} \leq C\norm{f}_{L^2(U; H^1_\gamma)}\,.$$
\end{lem}
\begin{proof}
Let $F(t) = \abs{t}\log(1+\abs{t})$ and $F^*(s) = \sup_{t\in\R} st - F(t)$ be its convex dual. We note that both $F$ and $F^*$ are nonnegative, convex, even, and strictly increasing on $[0,\infty)$. Additionally, $F(0) = F^*(0) = 0$ and 
$$\lim_{\abs{t}\to\infty} \abs{t}^{-1} F(t) = \lim_{\abs{s}\to\infty} \abs{s}^{-1}F^*(s) = \infty\,.$$
Therefore, the Orlicz spaces $L_F(U\times \Rd, dxd\gamma)$ and $L_{F^*}(U\times\Rd, dxd\gamma)$ defined by the norms 
\begin{align*}
\norm{f}_{L_F(U\times \Rd, dxd\gamma)} &= \inf\insb*{t>0 : \int_{U\times\Rd} F(t^{-1} g)\,dxd\gamma \leq F(1)} \\
\norm{f}_{L_{F^*}(U\times \Rd, dxd\gamma)} &= \inf\insb*{s>0 : \int_{U\times\Rd} F^*(s^{-1} g)\,dxd\gamma \leq F^*(1)}
\end{align*}
are dual Banach spaces and satisfy the following generalized H\"older inequality
$$\int_{U\times \Rd} \abs{fg}\,dxd\gamma \leq \norm{f}_{L_F(U\times \Rd, dxd\gamma)}\norm{g}_{L_{F^*}(U\times \Rd, dxd\gamma)}\,,$$
for all $f\in L_F(U\times \Rd, dxd\gamma)$ and $g\in L_{F^*}(U\times\Rd, dxd\gamma)$. See \cite[Proposition 3.3.1]{RR} for a proof of this fact.

Now, using the logarithmic Sobolev inequality, which states that, for $f\in H^1_\gamma$ with $\norm{f}_{L^2_\gamma} =1$, 
$$\int_{\Rd} f^2 \log(1 + f^2)\,d\gamma \leq C \int_{\Rd}\abs{\nabla f}^2\,d\gamma\,,$$
we see that, for $f\in L^2(U; H^1_\gamma)$,
$$\norm{f}_{L_F(U\times \Rd, dxd\gamma)} \leq C \norm{f}_{L^2(U; H^1_\gamma)}\,.$$
Additionally, as $s(t+1) \leq t\log(1+t) + e^s$ for $s,t>0$, we see that $F^*(s) \leq e^s -s$. In particular, this allows us to conclude that
$$\norm*{\abs{v}^2}_{L_{F^*}(U\times \Rd, dxd\gamma)} \leq C\,.$$
Combining the previous two displays and the generalized H\"older inequality given above yields the desired inequality.
\end{proof}

Finally, in the proof of \tref{perhypohom}, we make use of some bounds on the fundamental solution of \eref{hypoeq}. We begin with a Nash-Aronson type estimate. The proof of this result can be found in \cite{Nash}. 
\begin{prop}[Nash-Aronson Estimate]
\label{p.nasharonson}
Let $P(t,x,v,s,y,w)$ be the Green function for \eref{hypoeq}. Then there exists a constant $C>0$ such that for $A> \frac{256(4+2\norm{\nabla H}^2_\infty)}{3}$ and for all $t,s>0$ and  $x,y,v,w\in\Rd$,
$$P(t,x,v,s,y, w) \leq C t^{-2d}\exp\inp*{-\tfrac{3\abs{(x-y) +tw}^2}{At^3} - \tfrac{32\abs{v-w}^2 +32\abs{w}^2}{3At}} +C t^{-\frac{d}{2}}\exp\inp*{-\tfrac{\abs{x-y}^2}{8At}} \,.$$
\end{prop}

Additionally, we also make use of a bound on the $H^1_{\kin}$ seminorm of the fundamental solution. First, to simplify notation, we define
$$\Gamma_c(t,x) = t^{-\frac{d}{2}} \exp\inp*{- c\frac{\abs{x}^2}{t}}\,.$$ 
Then, we can apply \lref{tcaccioppoli} and \pref{nasharonson} to see that, for $x,v,z\in \Rd$
\begin{align}
\label{e.gradnasharonson}
\dashint_\frac{t}{2}^t \dashint_{B_{\sqrt{t}}(z)} \norm{\nabla_v &P(s,x,v,\cdot)}_{L^2_\gamma}^2 +\norm{(\partial_t-v\cdot\nabla_x)P(s,x,v,\cdot)}^2_{H^{-1}_\gamma}\,dxds \\
&\leq t^{-1} \norm{P(\cdot, x,v,\cdot)}_{L^\infty((\frac{t}{2}, t)\times B_{\sqrt{t}}(z))} \leq Ct^{-1}\Gamma_c(t, x-z)^2\,.\nonumber
\end{align}

\subsection{Time Independent Results}
In this subsection, we introduce the functional analytic framework used in section 5 to prove \tref{largescale}. Using the same measure definitions as in the previous section, we give similar definitions of our function space and weak solution as before, omitting the $t$ variable.

The space $H^1_\hyp(U)$ is defined by
\begin{equation*}  
H^1_\hyp(U) := \insb*{ f \in L^2\left(U;H^1_\gamma\right) \ : \ v \cdot \nabla_x f \in L^2(U;H^{-1}_\gamma)}\,,
\end{equation*}
and it is a Banach space with respect to the norm
\begin{equation*}
\left\| f \right\|_{H^1_\hyp(U)} 
:=
\left( \left\| f \right\|_{L^2(U;H^1_\gamma)}^2
+ 
\left\| v \cdot \nabla_x f \right\|_{L^2(U;H^{-1}_\gamma)}^2 \right)^{\frac12}\,.
\end{equation*}
We also define the seminorm
\begin{equation*}
\left\llbracket f \right\rrbracket_{H^1_\hyp(U)} 
:=
\left( \left\| \nabla_v f \right\|_{L^2(U;L^2_\gamma)}^2
+ 
\left\| v \cdot \nabla_x f \right\|_{L^2(U;H^{-1}_\gamma)}^2 \right)^{\frac12}
\end{equation*}
so that we have $\left\| f \right\|_{H^1_\hyp(U)}^2 = \left\llbracket f \right\rrbracket_{H^1_\hyp(U)}^2 + \left\| f \right\|_{L^2(U;L^2_\gamma)}^2$.

We say that $f \in H^1_\hyp(U)$ is a \emph{weak solution} of the equation
\begin{equation}
\label{e.hypoeqnot}
-\nabla_v\cdot\a\nabla_v f + v \cdot \a\nabla_v f - v \cdot \nabla_x f + \nabla H(x) \cdot \nabla_v f = 0 \quad \mbox{in} \ U \times \Rd
\end{equation}
provided that, for all $g\in L^2(U; H^1_\gamma)$
\begin{equation}
\label{e.weaksoldef1}
\int_{U\times \Rd} 
\nabla_vg\cdot \a\nabla_v f \,dm
=
\int_{U\times \Rd} 
g\left(
v\cdot \nabla_xf
- \nabla H\cdot \nabla_vf \right)\,dm\,.
\end{equation}
Using the notation of $\nabla^*_v$ and $B$ defined previously, we can rewrite \eref{weaksoldef1}  as follows
\begin{equation}
\label{e.weaksoldef2}
\forall g\in L^2(U;H^1_\gamma),  \quad
\int_{U\times \Rd} 
\left( \nabla_vg\cdot \a\nabla_v f + g Bf \right) \,dm
=
0\,.
\end{equation}
Consequently, we can rewrite the equation \eref{hypoeqnot} as 
$$\nabla_v^* \a\nabla_v f + Bf = 0\,.$$

\smallskip

Now that we have defined the major function space we will be working with, we state a few inequalities satisfied by elements of this space.
\begin{prop}[{Poincar\'e inequality for $H^1_\hyp$~\cite[Theorem 1.3]{AAMN}}]
\label{p.hypoelliptic.poincare}
Fix a bounded $C^{1,1}$ domain~$U\subseteq\Rd$. There exists~$C(d,U)<\infty$ such that, for every $f \in H^1_{\hyp} (U)$, 
\begin{equation*} 
\left\| f - (f)_U \right\|_{L^2 (U; L^2_\gamma )} 
\leq
C \left\llbracket f \right\rrbracket_{H^1_\hyp(U)}\,.
\end{equation*}
Moreover, if in addition we have $f\in H^1_{\hyp,0}(U)$, then
\begin{equation*} 
\left\| f \right\|_{L^2(U;L^2_\gamma)} 
\leq
C \left\llbracket f \right\rrbracket_{H^1_\hyp(U)}\,.
\end{equation*}
\end{prop}

The following lemma was essentially proved in \cite[Lemma 5.1]{AAMN}, but we use a slightly different statement, which is more explicit with the dependence in $r$. We therefore include the complete proof. 

\begin{lem}[Caccioppoli inequality]
\label{l.caccioppoli}
There exists a constant~$C(d,\Lambda)<\infty$ such that, for every $r>0$ and $f \in H^1_\hyp(B_{r})$ satisfying
\begin{equation}
\label{e.caccio.pde}
-\nabla_v\cdot\a\nabla_v f + v\cdot \a\nabla_v f - v\cdot \nabla_x f + \nabla H\cdot \nabla_v f = 0 \quad \mbox{in} \ B_r \times\Rd\,, 
\end{equation}
we have, for every $s\in (0,r)$, the estimate
\begin{equation*}
\left\llbracket f \right\rrbracket_{H^1_{\hyp}(B_{s}) }
\leq 
\frac{C}{r-s}
\left\| f - (f)_{B_{r}\setminus B_{s}} \right\|_{L^2(B_{r}\setminus B_{s};L^2_\gamma)}\,.
\end{equation*}
\end{lem}
\begin{proof}
By subtracting a constant, we may suppose that $(f)_{B_{r}\setminus B_{s}} =0$.  
Following the proof of \cite[Lemma 5.1]{AAMN}, we select $\phi\in C^\infty_c(B_r)$ satisfying
$$
\1_{B_s}\leq \phi \leq \1_{B_r}
\quad 
\text{and}
\quad
\abs{\nabla\phi} \leq C(r-s)^{-1}\, .
$$
Testing \eqref{e.weaksoldef2} of the equation with $\phi^2 f$ and using the ellipticity of $\a$, we see that
\begin{equation}
\label{e.caccio.plugin}
\int_{B_r\times\Rd} \phi^2 \left| \nabla_vf \right|^2 \,dm
\leq 
- \int_{B_r\times\Rd} \phi^2 f Bf \,dm\,. 
\end{equation} 
Integrating by parts, the right hand side becomes
\begin{align*}
-\int_{B_r\times\Rd} \phi^2 f Bf \,dm &= -\int_{B_r\times\Rd} 2\phi f \nabla\phi\cdot \nabla_v f \,dm \\
&\leq \frac12\int_{B_r\times\Rd} \phi^2 \left| \nabla_vf \right|^2 \,dm + 2\int_{B_r\times\Rd} \left|\nabla\phi \right|^2 f^2 \,dm\,,
\end{align*}
where the last step follows by applying Cauchy's inequality. Combining this with \eref{caccio.plugin} and using the properties of $\phi$, 
\begin{equation}
\label{e.caccgradbnd}
\int_{B_s\times\Rd} \abs{ \nabla_vf}^2 \,dm
\leq 
\frac{C}{(r-s)^2} \int_{( B_{r}\setminus B_s ) \times\Rd} 
f^2\,dm\,. 
\end{equation}
This yields the desired estimate of $\left\| \nabla_v f \right\|_{L^2( B_s; L^2_\gamma)}$. To estimate $\left\| v\cdot \nabla_x f \right\|_{L^2(B_s; H^{-1}_\gamma)}$, we test \eqref{e.weaksoldef2} with $g \exp(H)$, for $g\in L^2(B_s;H^1_\gamma)$, and use \eqref{e.caccgradbnd} and the upper bound on $\a$ to see that
\begin{align*}
\left| \int_{B_s\times\Rd} g (v\cdot \nabla_x f) 
\,dx\,d\gamma \right| 
& =
\left| \int_{B_s\times\Rd} \left( \a\nabla_vg +g \nabla H \right) \cdot \nabla_vf
\,dx\,d\gamma \right| 
\\ & 
\leq C \left\| g \right\|_{L^2(B_s;H^1_\gamma)} 
\left\| \nabla_v f \right\|_{L^2(B_s;L^2_\gamma)}
\\ & 
\leq 
\frac{C}{r-s} \left\| g \right\|_{L^2(B_s;H^1_\gamma)} 
\left\| f \right\|_{L^2(B_{r}\setminus B_s;L^2_\gamma)}\,.
\end{align*}
Taking the supremum over $g\in L^2(B_s;H^1_\gamma)$ with $\left\| g \right\|_{L^2(B_s;H^1_\gamma)}\leq 1$ yields 
\begin{equation*}
\left\| v\cdot \nabla_x f \right\|_{L^2(B_s;H^{-1}_\gamma)} \leq \frac{C}{r-s}\left\| f \right\|_{L^2(B_{r}\setminus B_s;L^2_\gamma)}\,. 
\end{equation*}
Combining \eqref{e.caccgradbnd} and the previous line yields
\begin{align*}
\snorm{f}_{H^1_{\hyp}(B_{s}) }^2
&=
\left\| \nabla_v f \right\|_{L^2(B_s; L^2_\gamma)}^2
+
\left\| v\cdot \nabla_x f \right\|_{L^2(B_s;H^{-1}_\gamma)}^2\\
&\leq  \frac{C}{(r-s)^2}\left\| f \right\|^2_{L^2(B_{r}\setminus B_s;L^2_\gamma)}\,,
\end{align*}
completing the proof.
\end{proof}
We note that, for $f$ a solution of \eref{caccio.pde}, $\inprod{Bf}_\gamma = \inprod{\nabla^*_v\nabla_v f}_\gamma = 0$. In this case, we can solve for $\g\in L^2(U; L^2_\gamma)^d$ such that $\nabla_v^* \g = Bf$ and $\norm{\g}_{L^2(U; L^2_\gamma)} \leq \norm{Bf}_{L^2(U; H^{-1}_\gamma)}$. In particular, this now let's us write, for $0<c<C<\infty$, 
$$c\snorm{f}_{H^1_\hyp(U)} \leq \norm{\nabla_v f}_{L^2(U; L^2_\gamma)} + \norm{\g}_{L^2(U; L^2_\gamma)} \leq C\snorm{f}_{H^1_\hyp(U)}\,.$$
In later section, when we need to integrate the $x$ variable with respect to different measures, this will simplify the resulting notation.

Additionally, we need the following bound on the $x$ regularity for functions in $H^1_\hyp(U)$. This result is proved in \cite[Theorem 1.4]{AAMN} and we provide the statement here for reference.
\begin{prop}[{H\"ormander inequality for $H^1_\hyp$ \cite[Theorem 1.4]{AAMN}}]
\label{p.hypoelliptic.hormander}
Let $\alpha\in [0, \frac{1}{3})$, and let $U = \Rd, \Td$. There exists a constant $C = C(d)< \infty$ such that, for every $f\in H^1_\hyp(U)$, we have the estimate
$$\norm{f}_{H^\alpha(U; L^2_\gamma)} \leq C \norm{f}_{H^1_\hyp(U)} \,.$$
For $\alpha = \frac{1}{3}$, we have the estimate
$$\norm{f}_{Q^{\sfrac{1}{3}}_{\nabla_x}(U)} \leq C\norm{f}_{H^1_\hyp(U)}\,.$$
\end{prop}

\section{Correctors and Heterogenous Polynomials}
The purpose of this section is to define the space of heterogenous polynomials $\A_m$ as well as provide several bounds that will be useful in the following sections. 

\subsection{Deriving the Correctors}
First, we discuss the existence and bounds for the correctors in this case. While we have so far been working with an unscaled version of the operator $L$, for the derivation of the correctors, it is a bit simpler to instead look at a scaled version of this operator, given by
$$L^\ep f = -\frac{1}{\ep^2} \nabla_v\cdot\a\nabla_v f + \frac{1}{\ep^2}v\cdot\a\nabla_v f -\frac{1}{\ep}v\cdot \nabla_x f + \frac{1}{\ep^2}\nabla H\inp*{\tfrac{x}{\ep}} \cdot\nabla_v f\,.$$
By looking at this operator in the theatrical scaling, it becomes much simpler to compare terms of the same order and deduce the form the correctors should take. We begin with a relatively informal derivation to provide some motivation for the following existence results.

In the spirit of the elliptic case, we take a function of the form
$$ w^\ep(x, v) = \sum_{n=0}^m \sum_{\abs{\alpha} = n} \ep^n \partial^\alpha f(x)  \phi_{\alpha}(\tfrac{x}{\ep}, v)\,,$$
where $\alpha\in\N^d$ is a multiindex, $\phi_0 = 1$, and the remaining $\phi_\alpha$ will be defined in the forthcoming lines. Applying the operator $L^\ep$ to $w^\ep$, we see that
\begin{align*}
L^\ep w^\ep &= \sum_{n=0}^m\sum_{\abs{\alpha}=n} \ep^{n-2} \partial^\alpha f L\phi_{\alpha} - \ep^{n-1} \phi_\alpha (v\cdot \nabla_x)\partial^\alpha f \\
&= \sum_{n=0}^m\sum_{\abs{\alpha} = n}\sum_{j=1}^d \ep^{n-2} \partial^\alpha f (L\phi_\alpha - v_j \phi_{\alpha-e_j}) +\sum_{\abs{\alpha} = m+1} \sum_{j=1}^d v_j \phi_{\alpha-e_j} \partial^{\alpha}f \,,
\end{align*}
where we set $\phi_{-e_j}=0$. This suggests that we should inductively construct the $\phi_{\alpha}$ by solving 
$$L\phi_{\alpha} = \sum_{j=1}^dv_j \phi_{\alpha-e_j} - \inprod{v_j \phi_{\alpha-e_j}}\,,$$
and set 
$$\ahom_{\alpha} = -\sum_{j=1}^d\inprod{v_j \phi_{\alpha-e_j}}\,,$$
and $\ahom_{\alpha} = 0$ for $\abs{\alpha}\leq 1$. We now prove the existence of the $\phi_\alpha$ as well as some basic bounds which will be used in the following sections.

\begin{lem}
\label{l.correctorexist}
For each $k\in\N$ and $\alpha\in\Nd$ such that $\abs{\alpha} = k$, there exists $\phi_\alpha\in H^1_\hyp(\Td)$ and $\ahom_\alpha$ such that for any $k^{\text{th}}$ degree polynomial, $p(x)$, we have 
\begin{equation}
\label{e.correctoreval}
\sum_{\abs{\alpha}\leq k}L\inp*{ \phi_{\alpha} \partial^\alpha p} = \sum_{\abs{\alpha}\leq k} \ahom_\alpha \partial^\alpha p\,.
\end{equation}
Additionally, there exists some $C(d, \Lambda)< \infty$ such that
\begin{equation}
\label{e.correctorbound}
\max_{\abs{\alpha} = k} \inp*{\norm{\phi_\alpha}_{H^1_\hyp(\Td)} + \abs{\ahom_\alpha}} \leq C^k\,.
\end{equation}
\end{lem}
\begin{proof}
Following the derivation above, we set $\phi_0 = 1$ and proceed inductively. Suppose, for some $k\in\N$, we have defined $\phi_\alpha$ for $\abs{\alpha} \leq k-1$ satisfying the required properties. We then define $\phi_\alpha$ with $\abs{\alpha} =k$ as the solution to the following equation
$$
\begin{cases}
-\nabla_v\cdot\a\nabla_v \phi_{\alpha} +v\cdot\a\nabla_v\phi_{\alpha}  - v\cdot\nabla_x\phi_{\alpha} + \nabla H \cdot\nabla_v \phi_{\alpha} = \sum_{j=1}^d v_j \phi_{\alpha-e_j} - \inprod{v_j\phi_{\alpha-e_j}} & \text{ in } \Td\times\Rd \\
\inprod{\phi_{\alpha}} = 0\,,
\end{cases}
$$
where we set $\phi_{\alpha - e_j} := 0$ if $\alpha\cdot e_j = 0$
Here, we can apply \cite[Theorem 1.2]{AAMN} to conclude that there exists a unique solution $\phi_\alpha\in H^1_\hyp(\Td)$ to the above equation. Furthermore, we have that 
\begin{align*}
\norm{\phi_\alpha}_{H^1_\hyp(\Td)} &\leq C\sum_{j=1}^d \norm{v_j \phi_{\alpha - e_j}}_{L^2(\Td; L^2_\gamma)} \\
&\leq C\sum_{j=1}^d \norm{ \phi_{\alpha - e_j}}_{L^2(\Td; H^1_\gamma)} \\
&\leq Cd C^{k-1} \leq C^k\,,
\end{align*}
for some $C(d, \Lambda) < \infty$, where we used \lref{orlicz} to pass from the first to the second line. Similarly, using the constructed $\phi_\alpha$, we define 
\begin{equation}
\label{e.ahomdef}
\ahom_{\alpha} = -\sum_{j=1}^d\inprod{v_j \phi_{\alpha-e_j}}\,.
\end{equation}
Following the same inductive argument presented above, we can prove the bounds for the $\ahom_\alpha$, completing the first part of the proof. 

Finally, we can compute, since $p$ has degree $k$,
\begin{align*}
\sum_{\abs{\alpha}\leq k}L\inp*{ \phi_{\alpha} \partial^\alpha p} &= \sum_{\abs{\alpha}\leq k} \partial^\alpha p L\phi_\alpha + \phi_\alpha L\partial^\alpha p \\
&= \sum_{\abs{\alpha}\leq k}\ahom_\alpha\partial^\alpha p + \sum_{j=1}^d v_j \phi_{\alpha - e_j} \partial^\alpha p - \phi_\alpha v\cdot\nabla_x\partial^\alpha p \\
&=\sum_{\abs{\alpha}\leq k}\ahom_\alpha\partial^\alpha p + \sum_{\abs{\alpha}=k} \sum_{i=1}^d\phi_\alpha v_i\partial^{\alpha+e_i}p \\
&= \sum_{\abs{\alpha}\leq k}\ahom_\alpha\partial^\alpha p\,,
\end{align*}
concluding the proof.
\end{proof}

\subsection{The Macroscopic Operator}
In this section, we define the macroscopic operator and prove that it is invertible on sufficiently large scales.

To begin, we first define several polynomial spaces that we will make use of throughout this section. Let $\P^*_m$ denote the set of polynomials on $\Rd$ of consisting only of terms of degree $m$. That is
$$\P^*_m = \insb*{p:\Rd\to \R \text{ ; } p(x) = \sum_{\abs{\alpha}= m} c_\alpha x^\alpha \text{ for } \alpha\in\Nd \text{ and } c_\alpha\in\R}\,.$$
We define $\P_m$ as 
$$\P_m = \insb*{p:\Rd\to \R \text{ ; } p(x) = \sum_{k=0}^m p_k \text{ for } p_k\in\P_k^*}$$
and let $\P = \bigcup_{m=0}^\infty \P_m.$
Furthermore, we define the set of harmonic polynomials, $\H$, by 
$$\H = \{p\in\P : \Delta p = 0\}.$$
Similarly, we define, for $m\in\N$ the spaces $\H_m = \H\cap \P_m$ and $\H^*_m = \H\cap \P^*_m$. Now, we consider the Laplacian as an operator on the space of polynomials. We begin by defining an inner product on $\P$ by 
$$\inprod{p,q}_{\P_r} = \sum_{\alpha\in\Nd}\frac{r^{2\abs{\alpha}}}{\alpha!} \partial^\alpha p(0) \partial^\alpha q(0)\,,$$
where $\alpha! = \prod_{i=1}^d (\alpha_i!)$ and $r>0$.
In particular, we can compute for $\alpha, \beta\in\Nd$, $\inprod{x^\alpha, x^\beta}_{\P_r} = \alpha!r^{2\abs{\alpha}} \1_{\alpha =\beta}$. From this, we can also compute 
\begin{align*}
\inprod{\abs{x}^2 x^\alpha, x^\beta}_{\P_r} &= \sum_{i=1}^d \inprod{x^{\alpha + 2e_i}, x^\beta}_{\P_r} \\
&= \sum_{i=1}^d \beta!r^{2\abs{\beta}}\1_{\alpha+2e_i = \beta} \\
&= \sum_{i=1}^d r^{4}\beta_i(\beta_i-1)\alpha!r^{2\abs{\alpha}} \1_{\alpha = \beta -2e_i} \\
&=\sum_{i=1}^d \inprod{x^{\alpha}, r^{4}\partial_{x_i}^2 x^\beta}_{\P_r} \\ 
&= \inprod{x^\alpha, r^{4}\Delta x^\beta}_{\P_r}\,.
\end{align*}
Therefore, we see that $r^{4}\Delta$ and multiplication by $\abs{x}^2$ are adjoints with respect to this inner product. In particular, as $r^4\Delta$ and $\Delta$ have the same kernel for $r>0$, it follows that $\P^*_m = \H^*_m \oplus \abs{x}^2\P^*_{m-2}$. Applying this inductively, we then see that 
\begin{equation}
\label{e.polydecomp}
\P^*_m = \bigoplus_{k=0}^{\floor{\sfrac{m}{2}}}\abs{x}^{2k}\H^*_{m-2k} \,.
\end{equation}
Now, for $p_k\in \H^*_{m-2k}$, we can compute that
$$\Delta(\abs{x}^{2k+2}p_k) = (2k+2)(d+2m-2k)\abs{x}^{2k}p_k\,.$$
From this computation and \eref{polydecomp}, we define the operator $S: \P\to \P$ by 
$$S(\abs{x}^{2k}p_k) = b_{k}^{-1}\abs{x}^{2k+2} p_k\,,$$
where $b_{k} = (2k+2)(d+2m-2k)$, and extend to $\P$ by linearity. In particular, the above computation shows that
$$\Delta Sp = p\,.$$
Combining the above computations, we see that $q = Sp$ is the unique polynomial such that $\Delta q = p$ and $q$ is orthogonal to all harmonic polynomials with respect to our inner product on $\P$. Additionally, we have the following bound on $Sp$.

\begin{lem}
\label{l.invlapbound}
There exists some $C(d)<\infty$ such that for $m\in \N$, $p\in\P^*_m$, and $r>0$, 
\begin{equation}
\label{e.invlapbound}
\norm{Sp}_{\P_r} \leq \frac{r^2}{\sqrt{m+1}}\norm{p}_{\P_r}\,.
\end{equation}
\end{lem}
\begin{proof}
Following \eref{polydecomp}, we can write 
$$p = \sum_{k=0}^{\floor{\sfrac{m}{2}}} \abs{x}^{2k} p_k\,,$$
for $p_k \in \H^*_{m-2k}$. Then, using that this decomposition is orthogonal, we see that
\begin{align*}
\norm{Sp}^2_{\P_r} &= \sum_{k=0}^{\floor{\sfrac{m}{2}}} b_{k}^{-2}\inprod{\abs{x}^{2k+2} p_k, \abs{x}^{2k+2} p_k}_{\P_r} \\
&= \sum_{k=0}^{\floor{\sfrac{m}{2}}} r^{4} b_{k}^{-1}\inprod{\abs{x}^{2k} p_k, b_{k}^{-1}\Delta(\abs{x}^{2k+2} p_k)}_{\P_r} \\ 
&= \sum_{k=0}^{\floor{\sfrac{m}{2}}} r^{4}b_{k}^{-1}\inprod{\abs{x}^{2k} p_k, \abs{x}^{2k} p_k}_{\P_r} \\
&\leq \frac{r^{4}}{4m+2d} \norm{p}_{\P_r}^2\,.
\end{align*}
Taking the square root and simplifying the denominator in the prefactor then gives the desired bound.
\end{proof}

Following \eref{correctoreval}, we define the macroscopic operator $\mc{A}:\P\to \P$ by
$$\mc{A}q = \sum_{\alpha\in\Nd} \ahom_\alpha \partial^\alpha q\,.$$
Note that, for $\abs{\alpha}\leq 1$, $\ahom_\alpha= 0$, so for $q\in \P_m$, $\mc{A}q \in \P_{m-2}$. 
We now want to show that this mapping is invertible. To do this, we will treat $\mc{A}$ as a perturbation of $\Delta$, and inductively use the bound in \lref{invlapbound} to derive a bound on the inverse, provided $r$ is sufficiently large.
\begin{lem}
\label{l.invertmacroop}
There exists $C(d,\Lambda)<\infty$ such that, for every $m\in\N$, and $p\in\P_m$ there exists $q\in \P_{m+2}$ satisfying
$$
\mc{A}q = p 
$$
and, for $r\geq Cm$,
$$\norm{q}_{\P_r} \leq Cr^2\norm{p}_{\P_r}\,.$$
In particular, we have that, for $k\in \{0, 1, \ldots, m\}$ and $r\geq Cm$,
\begin{equation*}
\abs{\nabla^{k+2} q(0)} \leq \sum_{n=k}^{m} r^{n-k} \abs{\nabla^n p(0)}\,.
\end{equation*}
\end{lem}
\begin{proof}
Up to a change of coordinates, we may take $\ahom := I_d$, the $d$-dimensional identity matrix. Then, we see that $\mc{A} = \Delta + \sum_{\abs{\alpha}\geq 3} \ahom_\alpha \partial^\alpha =: \Delta + \mc{A}'.$ 

We begin by considering $p\in \P_m^*$. We then recursively construct the sequence of polynomials $\{q_k\}$ as follows, 
$$\begin{cases}
q_0 = Sp \\
q_k = -S\inp*{\sum_{i=3}^{k+2}\sum_{\abs{\alpha}=i} \ahom_\alpha \partial^\alpha q_{k+2-i}}\,.
\end{cases}$$
Following an inductive argument, we see that $q_k \in \P^*_{m+2-k}$. Therefore, we can see that $q_k = 0$ for $k\geq m+2$. Now, we claim that $q = \sum_{k=0}^{m+1} q_k$ satisfies $\mc{A} q = p$ with the desired bound. Using the definition of $q_k$, we can compute that
\begin{align*}
\Delta q = p - \sum_{k=1}^{m+1} \sum_{i=3}^{k+2}\sum_{\abs{\alpha}=i} \ahom_\alpha \partial^\alpha q_{k+2-i} 
= p -\sum_{j=0}^{m+1}\sum_{l=3}^{m+2 - j}\sum_{\abs{\alpha}=l} \ahom_\alpha \partial^\alpha q_{j} = p - \mc{A}'q\,.
\end{align*}

Now it only remains to show that the desired bounds hold.
Applying \eref{invlapbound}, we see that
$$\norm{q_0}_{\P_r} \leq \frac{r^2}{\sqrt{m+1}}\norm{p}_{\P_r}\,,$$
Then, proceeding via induction, supposing that for $j\leq k-1$
\begin{equation}
\label{e.inductpolybound}
\norm{q_j}_{\P_r} \leq   \frac{r^2}{\sqrt{m+1}}\inp*{\frac{Cm}{r}}^j \norm{p}_{\P_r} \,,
\end{equation}
we can compute, using \eref{invlapbound} and \eref{correctorbound}, that
\begin{align*}
\norm{q_k}_{\P_r} &\leq r^2(m-k+1)^{-\frac{1}{2}} \sum_{i=3}^{k+2} \big\| \!\sum_{\abs{\alpha}=i}\ahom_\alpha \partial^\alpha q_{k+2-i} \big\|_{\P_r}\\
&\leq \frac{r^2}{\sqrt{m+1}} \sum_{i=3}^{k+2}C^i\inp*{\frac{(m-k+i)!}{(m-k)!}}^\frac{1}{2} r^{-i}\norm{q_{k+2-i}}_{\P_r} \\
&\leq \frac{r^2}{\sqrt{m+1}}\norm{p}_{\P_r}\sum_{i=3}^{k+2}  C^i\inp*{\frac{(m-k+i)!}{(m+1)(m-k)!}}^\frac{1}{2} r^{2-i}\inp*{\frac{Cm}{r}}^{k+2-i}\\
&\leq \frac{r^2}{\sqrt{m+1}}\inp*{\frac{Cm}{r}}^k \norm{p}_{\P_r}\sum_{i=3}^{k+2} m^{\frac{i+3}{2} -i} \\
&\leq \frac{r^2}{\sqrt{m+1}}\inp*{\frac{Cm}{r}}^k \norm{p}_{\P_r}\,.
\end{align*}
So summing this over $k\in\{0, \ldots, m+1\}$, we see that, for $r\geq Cm$
\begin{align*}
\norm{q}_{\P_r} \leq \sum_{k=0}^{m+1}\norm{q_k}_{\P_r} \leq \frac{r^2}{\sqrt{m+1}}\norm{p}_{\P_r}\sum_{k=0}^{m+1}\inp*{\frac{Cm}{r}}^k 
\leq \frac{Cr^2}{\sqrt{m+1}}\norm{p}_{\P_r}.
\end{align*}
Furthermore, from \eref{inductpolybound}, we see that, for $k\in \{0, 1, \ldots, m\}$,
\begin{align*}
\abs{\nabla^{k+2} q(0)} &= \frac{((k+2)!)^\frac{1}{2}}{r^{k+2}} \norm{q_{m-k}}_{\P_r} \\
&\leq \frac{r^2}{\sqrt{m+1}} \frac{(Cm)^{m-k}((k+2)!)^\frac{1}{2}}{r^{m+2}} \norm{p}_{\P_r} \\
&= \frac{(Cm)^{m-k}((k+2)!)^\frac{1}{2}}{((m+1)!)^\frac{1}{2}} \abs{\nabla^m p(0)} \\
&\leq  (Cm)^{m-k} \abs{\nabla^m p(0)} \\
&\leq r^{m-k}\abs{\nabla^m p(0)}\,.
\end{align*}
Finally, to extend this estimate to $p\in\P_m$, we write 
$$p = \sum_{k=0}^m p_k\,,$$
where $p_k\in \P_k^*$. Then, using the result proved above, we can construct $q_k\in \P_{k+2}$ such that $\mc{A}q_k = p_k$ and 
$$\norm{q_k}_{\P_r} \leq \frac{Cr^2}{\sqrt{k+1}}\norm{p_k}_{\P_r}\,.$$
Therefore, letting $q = \sum_{k=0}^m q_k$, we see that, by linearity, $\mc{A}q = p$ and, since the $p_k$ are orthogonal with respect to the $\P_r$ inner product, 
\begin{align*}
\norm{q}^2_{\P_r} &\leq \sum_{k=0}^m \norm{q_k}^2_{\P_r} \\
&\leq \sum_{k=0}^m \frac{Cr^4}{k+1}\norm{p_k}^2_{\P_r} \\
&\leq Cr^4 \norm{p}^2_{\P_r}\,.
\end{align*}
Finally, by using the bound obtained in the homogeneous case, we see that, for $k\in\{0, \ldots, m\}$,
\begin{align*}
\abs{\nabla^{k+2} q(0)} &\leq \sum_{n=k}^m \abs{\nabla^{k+2} q_n(0)} \\
&\leq \sum_{n=k}^m r^{n-k}\abs{\nabla^n p(0)}\,,
\end{align*}
completing the proof.
\end{proof}

\subsection{Heterogeneous Polynomials}
We define the space of heterogenous polynomials $\A_m$ as follows 
$$\A_m = \insb*{\psi : \psi = \sum_{\abs{\alpha}\leq m} \phi_\alpha \partial^\alpha p \text{ for } p\in \P_m}\,.$$
Notice that for $\psi\in\A_m$, with $\psi = \sum_{\abs{\alpha}\leq m}\phi_\alpha \partial^\alpha q$, we can rewrite 
$$q(x) = \sum_{k=0}^m \sum_{\abs{\alpha} = k} \frac{1}{\alpha!} x^\alpha\partial^\alpha q(0)$$ 
and use the bounds given in \lref{correctorexist} to see that
\begin{equation}
\label{e.easyhetpolybound}
\norm{\psi}_{\uL^2(Q_r;L^2_\gamma)} \leq \sum_{k=0}^m \inp*{\frac{Cr}{k+1}}^k \abs{\nabla^k q(0)}\,,
\end{equation}
for $r\geq 1$. Using this bound, we now prove an existence result for members of $\A_m$ with respect to the equation.
\begin{lem}
\label{l.hetpolysolveseq}
For $m\in\N$ and $p\in\P_m$, there exists $\psi\in\A_{m+2}$ such that
$$
-\Delta_v \psi +v\cdot\nabla_v\psi -v\cdot\nabla_x\psi +\nabla H\cdot\nabla_v\psi = p  
$$
and, for $r\geq Cm^2$, 
$$\norm{\psi}_{\uL^2(Q_r; L^2_\gamma)} \leq \sum_{n=0}^m  \inp*{\frac{C r}{n+3}}^{n+2} \abs{\nabla^n p(0)} \,.$$
\end{lem}
\begin{proof}
This result is essentially a corollary of \lref{invertmacroop}. Applying the aforementioned lemma, there exists $q\in\P_{m+2}$ such that $\mc{A}q = p$. Then, setting 
$$\psi = \sum_{\abs{\alpha} \leq m} \phi_\alpha \partial^\alpha q\,, $$
and applying \eref{correctoreval} gives that 
$$
-\Delta_v \psi +v\cdot\nabla_v\psi -v\cdot\nabla_x\psi +\nabla H\cdot\nabla_v\psi = p\,.
$$
Finally, applying the bounds from \eref{easyhetpolybound} and \lref{invertmacroop}, we see that, for $r\geq Cm^2$
\begin{align*}
\norm{\psi}_{\uL^2(Q_r; L^2_\gamma)} &\leq  \sum_{k=2}^{m+2} \inp*{\frac{Cr}{k+1}}^k \abs{\nabla^k q(0)} \\
&\leq \sum_{k=0}^m \sum_{n=k}^m \frac{C^k r^{n+2}}{(k+3)^{k+2}} \abs{\nabla^n p(0)} \\
&\leq \sum_{n=0}^m r^{n+2}\sum_{k=0}^n \frac{C^k}{(k+3)^{k+2}} \abs{\nabla^n p(0)} \\
&\leq \sum_{n=0}^m  \inp*{\frac{C r}{n+3}}^{n+2} \abs{\nabla^n p(0)} \,,
\end{align*}
completing the proof
\end{proof}
Now, for the following sections, rather than simply having bounds on  $\psi\in \A_m$ in terms of the element of $\P_m$ defining it, it will also be useful to have bounds in terms of averages of the finite differences of $\psi$. To begin, we define the first order finite differences in $x$, $D^if $ by 
$$D^if(x,v)= f(x+e_i,v) - f(x,v)\,.$$
Likewise, for some $\alpha\in\Nd$, we define the higher order finite differences by
$$D^\alpha f(x,v) = (D^{\alpha_1}_1 D^{\al_2}_2 \ldots D^{\al_d}_d f)(x,v)\,.$$
In addition, we define the averages of $f$ over the lattice, $\hat{f}(x)$, as follows
$$\hat{f}(x):= \int_{(z+Q_1)\times\Rd} f(y,v)\,dyd\gamma(v)\,,$$
where $z\in\Zd$ is such that $x\in z+Q_1$.

The following lemma and proof follow \cite[Lemma 2.7]{AKS} very closely, but we reproduce the entire proof here for completeness.
\begin{lem}
\label{l.hetpolybound}
There exists $C(d, \Lambda) <\infty$ such that for $m\in\N$ and $\psi\in \A_m$ such that
$$\psi(x) = \sum_{\abs{\alpha} \leq m} \phi_\alpha \partial^\alpha q $$
for $q\in \P_m$, we have that, for $n\leq m$,
\begin{equation}
\label{e.polytermbound}
\inp*{\frac{m}{n+1}}^n \abs{\nabla^n q(0)} \leq \sum_{k=n}^m \inp*{\frac{Cm}{k+1}}^k \abs{D^k \hat{\psi}(0)}\,.
\end{equation}
In particular, for $r\geq m$, 
\begin{equation*}
\norm{\psi}_{\uL^2(Q_r; L^2_\gamma)} \leq \sum_{k=0}^m \inp*{\frac{Cr}{k+1}}^k \abs{D^k \hat{\psi}(0)}\,.
\end{equation*}
\end{lem}
\begin{proof}
We begin by proving \eref{polytermbound} through induction. An easy computation shows that \eref{polytermbound} holds in the case that $m=0$ or $m=1$. Now, suppose $m\geq 2$ and $q\in\P_{m-1}$ we define $\psi\in\A_{m-1}$ by 
$$\psi(x) = \sum_{\abs{\alpha} \leq m-1} \phi_\alpha \partial^\alpha q$$
and suppose the following inequality holds
\begin{equation*}
\inp*{\frac{C(m-1)}{n+1}}^n \abs{\nabla^n q(0)} \leq \sum_{k=n}^{m-1} \inp*{\frac{C(m-1)}{k+1}}^k \abs{D^k \hat{\psi}(0)}\,,
\end{equation*}
for $n\leq m-1$.

Taking $q\in\P_m$ and defining $\psi$ as before, we define $q_m\in \P_m^*$ to be the homogeneous polynomial consisting of the terms of $q$ of order $m$. That is to say
$$D^m q_m = D^m q, \quad D^kq_m(0) = 0 \text{ for }k\in\{0, 1, \ldots, m-1\}\,.$$
As each $\phi_\alpha$ is $\Zd$-periodic and $q\in\P_m$, we see that, for $x\in Q_1(z)$,
$$D^m \hat{\psi}(x) = \dashint_{Q_1(z)} \sum_{\abs{\alpha} \leq m} \phi_\alpha(x) D^m \partial^\alpha q(x) = D^m q(x)\,.$$
Furthermore, as $q_m\in \P^*_m$, we can compute that $D^m q_m(0) = \nabla^m q_m(0)$, which, in particular, implies that 
$$q_m(x) = \sum_{\abs{\alpha} = m} \tfrac{1}{\alpha!}D^\alpha q_m(0) x^\alpha\,.$$ 
Using this representation, we can compute that for $n,k\in\N$ such that $n+k\leq m$ and $x\in\Rd$, 
\begin{align}
\label{e.derivdiffbound}
\abs{\nabla^k D^n q_m(x)} &\leq \frac{C^{m-n-k}}{(m-n-k)!} \abs{D^m q_m(0)} (k+\abs{x})^{m-n-k} \\
&= \frac{C^{m-n-k}}{(m-n-k)!} \abs{D^m \hat{\psi}(0)} (k+\abs{x})^{m-n-k}\,.\nonumber
\end{align}
So, defining 
$$\chi = \sum_{\abs{\alpha}\leq m} \phi_\alpha \partial^\alpha q_m\,,$$
we can use Jensen and the bound above to compute that, for $n\leq m$
\begin{align*}
\abs{D^n \hat{\chi}\;(0)} &\leq \norm{D^n \chi}_{\uL^2(Q_1; L^2_\gamma)} \\
&\leq \sum_{\abs{\alpha} \leq m-n} \norm{\phi_\alpha}_{\uL^2(Q_1; L^2_\gamma)} \norm{\partial^\alpha D^n q_m}_{\uL^2(Q_1)}  \\
&\leq \abs{D^m\hat{\psi}(0)} \sum_{k=0}^{m-n} C^k \inp*{\frac{C(k+1)}{(m-n-k)!}}^{m-n-k} \\
&\leq \abs{D^m\hat{\psi}(0)} C^{m-n}\,.
\end{align*}
Then, since $\psi - \chi \in \A_{m-1}$, we can apply the inductive hypothesis, as well as the previous two bounds to see that, for $n\leq m-1$,
\begin{align*}
\inp*{\frac{m}{n+1}}^n \abs{\nabla^n q(0)} &\leq \inp*{\frac{m}{n+1}}^n \abs{\nabla^n (q-q_m)(0)} + \inp*{\frac{m}{n+1}}^n \abs{\nabla^n q_m(0)} \\
&\leq \inp*{\frac{m}{m-1}}^n \sum_{k=n}^{m-1} \inp*{\frac{C(m-1)}{k+1}}^k (\abs{D^k \hat{\psi}(0)}+ \abs{D^k\hat{\chi}(0)}) \\
&\quad+ \inp*{\frac{m}{n+1}}^n\inp*{\frac{Cn}{m-n}}^{m-n} \abs{D^m\hat{\psi}(0)} \\
&\leq \sum_{k=n}^{m-1} \inp*{\frac{Cm}{k+1}}^k \abs{D^k \hat{\psi}(0)} + \abs{D^m \hat{\psi}(0)} \sum_{k=n}^{m-1} C^{m-k}\inp*{\frac{Cm}{k+1}}^k \\
&\quad+ \frac{1}{2}\inp*{\frac{Cm}{m+1}}^m\abs{D^m\hat{\psi}(0)} \\
&\leq \sum_{k=n}^{m} \inp*{\frac{Cm}{k+1}}^k \abs{D^k \hat{\psi}(0)} \,,
\end{align*}
where, to conclude, we note that, provided $C$ is sufficiently large,
$$\sum_{k=n}^{m-1} C^{m-k}\inp*{\frac{Cm}{k+1}}^k  \leq \frac{1}{2}\inp*{\frac{Cm}{m+1}}^m\,.$$
Finally, applying this bound to \eref{easyhetpolybound}, we see that for $r\geq m$, 
\begin{align*}
\norm{\psi}_{\uL^2(U; L^2_\gamma)} &\leq \sum_{k=0}^m \inp*{\frac{Cr}{k+1}}^k \abs{\nabla^k q(0)} \\
&\leq \sum_{k=0}^m \sum_{n=k}^m \inp*{\frac{Cr}{m}}^k\inp*{\frac{Cm}{n+1}}^n \abs{D^n \hat{\psi}(0)} \\
&= \sum_{n=0}^m \sum_{k=0}^n \inp*{\frac{Cr}{m}}^k\inp*{\frac{Cm}{n+1}}^n \abs{D^n \hat{\psi}(0)} \\
&\leq \sum_{n=0}^m\inp*{\frac{Cr}{m}}^n\inp*{\frac{Cm}{n+1}}^n \abs{D^n \hat{\psi}(0)} \\
&= \sum_{n=0}^m \inp*{\frac{Cr}{n+1}}^n \abs{D^n \hat{\psi}(0)}\,,
\end{align*}
where the second to last inequality holds provided $C$ is sufficiently large.
\end{proof}

Using this bound and inducting on $m$, we can see that there exists a unique $\psi\in\A_m$ with prescribed values of $D^k\hat\psi(0)$ for $k\leq m$.

\section{Homogenization of the Cauchy Problem}
In this section, we give the proof of \tref{perhypohom}. To simplify notation, we let $y,w=0$. Additionally, we fix $t$ sufficiently large corresponding to the time at which we want to estimate $P-\Phom$. In order to estimate this difference, we will pick some $t_0\in[1,\frac{1}{2}t]$ to define an mesoscopic layer from which we will begin building our approximate solutions. However, $P$ and $\Phom$ will almost certainly not be equal at time $t_0$, which can make some of our later computations more complicated. To help fix this issue, we introduce a new reference solution centered at $t_0$ and compare $P$ and $\Phom$ to this new function. We define $\Qhom$ as the solution to the following initial value problem

\begin{equation}
\label{e.homogenouseq}
\begin{cases}
\partial_t \Qhom - \nabla\cdot\ahom\nabla \Qhom = 0 & \text{ in } (t_0,\infty)\times \Rd \\
\Qhom(t_0, \cdot) = \inprod{P(t_0, \cdot, 0)}_\gamma & \text{ on } \Rd \, .
\end{cases}
\end{equation}

We first note that $\Qhom$ and $\Phom$ are close. In particular, up to rescaling, $\Qhom(t,x)- \Phom(t,x)$ is a solution to the heat equation with an initial condition that has zero mean. This implies that this difference decays to zero faster than a typical solution to the heat equation. In particular, by writing this difference as a convolution and Taylor expanding heat kernel, we see that
$$\abs{\Qhom(t,x) - \Phom(t,x)} \leq C \inp*{\frac{t_0}{t}}^{\frac{1}{2}} \Gamma_c(t,x)\,.$$

Now it only remains to estimate $P- \Qhom$. To do this, we proceed via a two-scale expansion argument and then apply Duhamel's principle to estimate this difference. To start, we let $\zeta(t)$ be a smooth cutoff function such that
$$\1_{(3t_0, t-2t_0)} \leq \zeta \leq \1_{(2t_0, t-t_0)} , \qquad \norm{\zeta'}_{L^\infty} \leq (2t_0)^{-1}\,.$$
Now, we define our two scale expansion as follows:
$$W(t,x,v) = \Qhom +  \zeta(t)\sum_{i=1}^d \partial_{x_i}\Qhom(t,x) \phi_{i}(x,v) + \zeta(t)\sum_{i,j=1}^d\partial_{x_i}\partial_{x_j}\Qhom(t,x) \psi_{ij}(x,v)  \,.$$
Here, $\phi_i$ solves the equation 
\begin{equation}
\label{e.correctoreq}
\begin{cases}
-\nabla_v\cdot\a\nabla_v \phi_i+ v\cdot\a\nabla_v \phi_i
- v\cdot\nabla_x \phi_i + \nabla H(x)\cdot \nabla_v \phi_i
= v_i& \text{ in } \T^d\times \R^d \\
(\phi_i)_{\T^d\times \Rd_\gamma} = 0\,,
\end{cases}
\end{equation}
and $\psi_{ij}$ solves
\begin{equation}
\label{e.secondcorrectoreq}
\begin{cases}
-\nabla_v\cdot\a\nabla_v \psi_{ij}+ v\cdot\a\nabla_v \psi_{ij}
- v\cdot\nabla_x \psi_{ij} + \nabla H(x)\cdot \nabla_v \psi_{ij}
= \ahom_{ij}(x,v)- \ahom_{ij} \\
(\psi_{ij})_{\T^d\times \Rd_\gamma} = 0\,,
\end{cases}
\end{equation}
where 
\begin{align*}
\ahom_{ij}(x,v) &= v_i\phi_j(x,v)\,,\\
\ahom_{ij}(x) &= \inprod{\ahom_{ij}(x,\cdot)}_\gamma\,, \\
\ahom_{ij} &= \inprod{\ahom_{ij}(\cdot)}_{\sigma}\,.
\end{align*}
The existence of $\phi_i$ and $\psi_{ij}$ as well as the derivation of their equations is presented in \lref{correctorexist} and the preceding discussion. Note that, using the equation $\psi_{ij}$ satisfies, it follows that
\begin{equation}
\label{e.avgpsieq}
\inprod{v\cdot\nabla_x \psi_{ij} + \nabla H \cdot \nabla_v \psi_{ij}}_\gamma = \ahom_{ij}(x) - \ahom_{ij}\,.
\end{equation}
and 
$$\norm{\psi_{ij}}_{L^2(\Td; H^1_\gamma)} \leq C \norm{\ahom(x)-\ahom}_{L^2(\Td)}\,.$$

Continuing, we now compute the equation that the two scale expansion, $W$, satisfies. In particular, we want to show that
\begin{align}
\label{e.correctorerror}
(\partial_t + L) W &=  - (\ahom(x) -\ahom):\nabla^2 \Qhom  - (1-\zeta) (v\cdot\nabla_x\Qhom - \partial_t\Qhom)\\
&\quad+\sum_{i=1}^d\phi_i \partial_t(\zeta\partial_{x_i}\Qhom) + \sum_{i,j=1}^d \psi_{ij}\partial_t(\zeta\partial_{x_i}\partial_{x_j}\Qhom) - \sum_{i,j,k=1}^d v_k\psi_{ij} \zeta\partial_{x_i}\partial_{x_j}\partial_{x_k} \Qhom\,.\nonumber
\end{align}
Using \eref{homogenouseq}, \eref{correctoreq}, and \eref{secondcorrectoreq}, we can see that
\begin{align*}
(\partial_t + L) W(t,x,v) 
&= \partial_t\Qhom -  v\cdot\nabla_x \Qhom 
+\sum_{i=1}^d L\phi_i\zeta\partial_{x_i}\Qhom 
- \phi_i \zeta(v\cdot \nabla_x\partial_{x_i} \Qhom) 
+  \phi_i \partial_t(\zeta\partial_{x_i} \Qhom)\\
&\quad+  \sum_{i,j=1}^d   L \psi_{ij} \zeta\partial_{x_i}\partial_{x_j} \Qhom 
- \zeta\psi_{ij}(v\cdot\nabla_x\partial_{x_i}\partial_{x_j} \Qhom) + \psi_{ij}\partial_t(\zeta \partial_{x_i}\partial_{x_j} \Qhom)   \\
&=-\sum_{i,j=1}^d\ahom_{ij}(x) \partial_{x_i}\partial_{x_j} \Qhom  - (1-\zeta)(v\cdot\nabla_x\Qhom - \ahom:\nabla^2\Qhom)\\
&\quad+\sum_{i=1}^d\phi_i \partial_t(\zeta\partial_{x_i}\Qhom) + \sum_{i,j=1}^d \psi_{ij}\partial_t(\zeta\partial_{x_i}\partial_{x_j}\Qhom) - \sum_{i,j,k=1}^d v_k\psi_{ij} \zeta\partial_{x_i}\partial_{x_j}\partial_{x_k} \Qhom \,,
\end{align*}
which yields the desired form given in \eref{correctorerror}. 

As an aside, it may seem a little strange that the first term in our two-scale expansion computed in the proof of \tref{perhypohom} is in non-divergence when typically, we want this term to more closely mirror the form of the homogenized equation. In fact, this term can be written in divergence form, which we can compute as follows:
\begin{align*}
\sum_{i,j = 1}^d \ahom_{ij}(x) \partial_{x_i}\partial_{x_j} \Qhom = \nabla\cdot\ahom(x)\nabla \Qhom + \sum_{i,j=1}^d \partial_{x_i}\ahom_{ij}(x) \partial_{x_j} \Qhom 
&= \nabla\cdot\ahom(x)\nabla \Qhom + \sum_{j=1}^d \partial_{x_j} \Qhom \inprod{v\cdot \nabla_x \phi_j}_\gamma \\
&= \nabla\cdot\ahom(x)\nabla \Qhom- \nabla H \cdot\ahom(x)\nabla \Qhom \, \\
&= e^{H(x)}\nabla\cdot(e^{-H(x)} \ahom(x)\nabla \Qhom) \\
&= e^{H(x)}\nabla\cdot((e^{-H(x)} \ahom(x) -\ahom)\nabla \Qhom). \\
\end{align*}
However, due to difficulties obtaining good estimates on the regularity in $x$ of solutions to \eref{hypoeq}, this error term is much more difficult to bound. Therefore, we use the error term in non-divergence form to prove convergence. 

So, using Duhamel's Principle, we see that
\begin{align}
\label{e.duhamel}
&(P-W)(t,x,v) \\
&\quad= \sum_{i,j=1}^d\int_{t_0}^t \zeta(s)\int_{\Rd} \inprod{P(t-s, x,v,z,\cdot)}_\gamma  (\ahom_{ij}(z) -\ahom_{ij})\partial_{x_i}\partial_{x_j} \Qhom \,dzds \nonumber\\ 
&\qquad + \int_{t_0}^t (1-\zeta(s))\int_{\Rd} \inprod{\nabla_v P(t-s, x,v,z,\cdot)}_\gamma\cdot\nabla_x \Qhom\,dzds\nonumber \\
&\qquad + \int_{t_0}^t (1-\zeta(s))\int_{\Rd} \inprod{P(t-s, x,v,z,\cdot)}_\gamma\ahom:\nabla^2_x \Qhom\,dzds\nonumber \\
&\qquad + \int_{\Rd\times\Rd} P(t-t_0, x,v,z,w)(P(t_0, z,w, 0, 0) - \inprod{P(t_0, z,w,0,0)}_\gamma)\,dm(z,w) \nonumber\\
&\qquad+ \sum_{i=1}^d\int_{t_0}^t\int_{\Rd\times\Rd} P(t-s, x,v,z,w) \phi_i (\partial_t(\zeta(s)\partial_{x_i}\Qhom))\,dm(z,w)ds \nonumber\\
&\qquad+\sum_{i,j=1}^d\int_{t_0}^t \zeta(s)\int_{\Rd\times\Rd} P(t-s, x,v,z,w) (\psi_{ij}\partial_t(\zeta(s)\partial_{x_i}\partial_{x_j}\Qhom))\,dm(z,w)ds \nonumber\\
&\qquad+\sum_{i,j,k=1}^d\int_{t_0}^t \zeta(s)\int_{\Rd\times\Rd} P(t-s, x,v,z,w) (v_k\psi_{ij} \partial_{x_i}\partial_{x_j}\partial_{x_k} \Qhom)\,dm(z,w)ds \,.
\nonumber 
\end{align}
Now, to conclude, it only remains to estimate 
$$\norm{P-W}_{\uL^2(\mc{Q}_T; L^2_\gamma)}\,.$$ 

To do this, we will bound each of these terms on the right hand side individually. Before doing so, we first prove two estimates that will be critical for these later bounds. The first of these gives us bounds on the time decay of $\Qhom$ and its derivatives relative to the heat kernel: for $s\geq 2t_0$
\begin{equation}
\label{e.Qhombound}
s^{\frac{3}{2}}\abs{\partial_t \nabla \Qhom(s,z)} + s^2\abs{\partial_t \nabla^2 \Qhom(s,z)} + \sum_{k=0}^3s^{\frac{k}{2}}\abs{\nabla^k \Qhom(s,z)} \leq C \Gamma_c(s,z) \,.
\end{equation}
To simplify the notation, we prove that 
$$s^{\frac{1}{2}}\abs{\nabla \Qhom(s,z)} \leq C \Gamma_c(s,z)\,.$$
The remaining bounds will follow from a very similar computation. Using the definition of $\Qhom$ and \pref{nasharonson}, we can compute that, 
\begin{align*}
s^{\frac{1}{2}}\abs{ \nabla &\Qhom(s,z)} \\
&\leq s^\frac{1}{2}  \int_{\Rd} \abs{\inprod{P(t_0, y, \cdot, 0)}_\gamma} \abs{\nabla \Phom(s-t_0, z-y)}\,dy \\
&\leq Cs^{\frac{1}{2}}\int_{\Rd} \Gamma_c(t_0, y) \abs{\nabla \Gamma_c(s-t_0, z-y)}\,dy \\
&= Cs^{\frac{1}{2}}(s-t_0)^{-1}\int_{\Rd} \abs{z-y}\Gamma_c(t_0, y) \Gamma_c(s-t_0, z-y)\,dy \\
&\leq Cs^{\frac{1}{2}}(s-t_0)^{-1} \inp*{\int_{\Rd} \abs{z-y}^2\Gamma_c(s-t_0, z-y)\,dy}^\frac{1}{2}\inp*{\int_{\Rd}\Gamma_c(t_0, y)^2 \Gamma_c(s-t_0, z-y)\,dy}^\frac{1}{2} \\
&\leq Cs^{\frac{1}{2}}(s-t_0)^{-\frac{1}{2}}\Gamma_c(s,z) \\
&\leq C\Gamma_c(s,z)\,.
\end{align*}
In the final line, we use the bound on $s$ to reduce all of the prefactors to a constant, completing the proof.

Additionally, we will bound the correctors in terms of the error stated in \tref{perhypohom}. Particularly, we will show that
\begin{equation}
\label{e.correctorerrorbound}
\int_{\Rd} (\norm{\phi_i}^2_{H^1_\gamma} + \norm{\psi_{ij}}^2_{H^1_\gamma}) \Gamma_c(s,z)\,dz \leq C \,.
\end{equation}
To prove this, we show that
$$\int_{\Rd} \norm{\phi_{i}}^2_{H^1_\gamma}\Gamma_c(s,z)\,dz \leq C\,.$$
The bound on the other term follows from the same argument, so we omit it here.
We first note that, defining $Q_r = [-\frac{1}{2}r, \frac{1}{2}r]$ as in section 1 and using that $B_{\frac{r}{2}}\subset Q_r \subset B_r$, it suffices to show that
$$
\int_{\Rd} \norm{\phi_{i}}^2_{H^1_\gamma}\Gamma_c(s,z)\,dz \leq C\dashint_{Q_{\sqrt{s}}\times \Rd} \phi_{i}^2\,dxd\gamma \,.
$$
Now, to prove the bound, we partition $\Rd$ into squares of unit length and use the periodicity of $\phi_{e_k}$ to see that
\begin{align*}
\int_{\Rd} \norm{\phi_{i}}^2_{H^1_\gamma}\Gamma_c(s,z)\,dz &= \sum_{z\in\Zd} \int_{Q_1(z)}\norm{\phi_{i}}^2_{H^1_\gamma}\Gamma_c(s,z)\,dz\\
&\leq \norm{\phi_{i}}^2_{L^2(Q_1; H^1_\gamma)} \sum_{z\in\Zd} \sup_{Q_1(z)}\Gamma_c(s,\cdot) \\
&\leq C\norm{\phi_{i}}^2_{L^2(Q_1; H^1_\gamma)}\\
&\leq C\,,
\end{align*}
as required.

Now, we'll bound each of these terms on the right hand side of \eref{duhamel} in order, starting with the first.

\subsubsection*{Term 1 bound}
We begin by taking the first term in \eref{duhamel}:
\begin{align*}
\int_{t_0}^t \zeta(s)\int_{\Rd} \inprod{P(t-s, x,v,z,\cdot)}_\gamma  (\ahom(z)-\ahom) :\nabla^2\Qhom \,dzds\,.
\end{align*}
We'll begin by bounding this convolution in $L^\infty$. Using \pref{nasharonson}, \eref{Qhombound}, and \eref{correctorerrorbound}, as well as H\"older's inequality, we see that
\begin{align*}
&\int_{t_0}^t\zeta(s) \int_{\Rd} \inprod{P(t-s, x,v,z,\cdot)}_\gamma (\ahom(z)-\ahom) :\nabla^2 \Qhom\,dzds\\
&\quad\leq \int_{2t_0}^{t-t_0} s^{-1}\inp*{\int_{\Rd} \frac{\inprod{P(t-s, x,v,z,w)}_\gamma^2}{\Gamma_c(t-s, x-z)}}^\frac{1}{2}\inp*{\int_{\Rd}\abs{\inprod{\nabla_v \phi_{e}}_\gamma}^2\Gamma_c(t-s, x-z) \Gamma_c(s,z)^2}^\frac{1}{2} \\
&\quad\leq \int_{2t_0}^{t-t_0}  s^{-1}\inp*{\int_{\Rd}\abs{\inprod{\nabla_v \phi_{e}}_\gamma}^2 \Gamma_c(t-s, x-z) \Gamma_c(s,z)^2} \\
&\quad\leq Ct^{\frac{d}{2}} \Gamma_c(t,x)\int_{2t_0}^{t-t_0} (t-s)^{-\frac{d}{4}} s^{-\frac{d}{4}-1}\inp*{\int_{\Rd}\abs{\inprod{\nabla_v \phi_{e}}_\gamma}^2 \Gamma_c(s,z)}^\frac{1}{2}\,ds \\
&\quad\leq C  t^{\frac{d}{2}} \Gamma_c(t,x) \int_{t_0}^{t-t_0} (t-s)^{-\frac{d}{4}}s^{-\frac{d}{4} -1}\,ds \\
&\quad\leq C\inp*{\frac{t}{t_0}}^{\frac{d}{4}}t_0^{-\frac{1}{2}}\Gamma_c(t,x) \,.
\end{align*}
Here, moving from the first to second line, we note that, using integration by parts, $\ahom(z) - \ahom = \inprod{\nabla_v \phi_e}_\gamma$. Additionally, moving from line 3 to 4, we again use the following bound on the heat kernel:
$$(t-s)^{\frac{d}{2}}\Gamma_c(t-s, x-z) s^{\frac{d}{2}}\Gamma_c(s,z) \leq Ct^d \Gamma_c(t,x)^2\,,$$
or, equivalently, that
$$\exp\inp*{-\frac{\abs{x-z}^2}{t-s} - \frac{\abs{z}^2}{s} + \frac{\abs{x}^2}{t}} \leq C\,.$$
This is a consequence of the triangle inequality as well as the monotonicity of the heat kernel. To conclude, we note that 
$$\dashint_{\mc{Q}_T} \inp*{\frac{t}{t_0}}^{\frac{d}{2}}t_0^{-1}\Gamma_c(t,x)^2 \leq Ct_0^{-\frac{d}{2}-1}\,.$$

\subsubsection*{Term 2 bounds}
Here, we begin by looking at the term
$$\dashint_{\mc{Q}_T}\inp*{\int_{t_0}^t (1-\zeta(s))\int_{\Rd} \inprod{\nabla_v P(t-s, x,v,z,w)}_\gamma\cdot \nabla\Qhom(s,z)\,dzds}^2\,.$$
To bound this term, we split the inner time integral into two intervals, $[t_0, 4t_0]$ and $[t-2t_0, t]$. For the first interval, we'll proceed similarly to the previous term, using \eref{gradnasharonson} as well. Using this bound and H\"older's inequality, we see that
\begin{align*}
&\dashint_{\mc{Q}_T}\inp*{\int_{t_0}^t (1-\zeta(s))\int_{\Rd} \inprod{\nabla_v P(t-s, x,v,z,w)}_\gamma\cdot \nabla\Qhom(s,z)\,dzds}^2\,dt\\
&\quad\leq \dashint_{\mc{Q}_T} \inp*{\int_{t_0}^{4t_0}\int_{\Rd} \norm{\nabla_v P(t-s, x,v,z,w)}_{L^2_\gamma}^2}\inp*{\int_{t_0}^{4t_0}\int_{\Rd}s^{-1}\Gamma_c(s,z)^2} \\
&\quad\leq \dashint_{\mc{Q}_T} \inp*{\sum_{z\in (\sqrt{t_0})\Zd}\int_{t_0}^{4t_0}\int_{B_{\sqrt{t_0}(z)}} \norm{\nabla_v P(t-s, x,v,z,w)}_{L^2_\gamma}^2}\inp*{\int_{t_0}^{4t_0}s^{-\frac{d}{2}-1}} \\
&\quad\leq C\dashint_{\mc{Q}_T} \inp*{\sum_{z\in (\sqrt{t_0})\Zd} \Gamma_c(t-t_0,x-z)^2} \\
&\quad\leq CT^{-\frac{d}{2}} \,.
\end{align*}
Next, to bound this integral on $[t-2t_0, t]$, we perform a similar computation to the bound of the first interval. Using \eref{gradnasharonson} and \eref{Qhombound}, we see that
\begin{align*}
&\dashint_{\mc{Q}_T}\inp*{\int_{t-2t_0}^{t} (1-\zeta(s))\int_{\Rd} \inprod{\nabla_v P(t-s, x,v,z,w)}_\gamma\cdot \nabla\Qhom(s,z)\,dzds}^2\,dt \\
&\quad\leq C\dashint_{\mc{Q}_T}\inp*{\sum_{z\in (\sqrt{t_0})\Zd}\int_{t_0}^{2t_0}\int_{B_{\sqrt{t_0}(z)}} \norm{\nabla_v P(s, x,v,z,w)}_{L^2_\gamma}^2}\inp*{\int_{t-2t_0}^{t} s^{-\frac{d}{2}-1}ds}\,dt \\
&\quad\leq C\dashint_{\mc{Q}_T} t^{\frac{d}{2}+1}_0 t^{-\frac{d}{2}-1}\inp*{\sum_{z\in (\sqrt{t_0})\Zd} \Gamma_c(t_0,x-z)^2} \\
&\quad\leq C T^{-1} t_0^{-\frac{d}{2}+1} \,.
\end{align*}

\subsubsection*{Term 3 bound}
Now, we will bound the term 
$$\dashint_{\mc{Q}_T}\inp*{\int_{t_0}^t (1-\zeta(s))\int_{\Rd} \inprod{P(t-s, x,v,z,w)}_\gamma\ahom:\nabla^2_x \Qhom\,dzds}^2\,.$$
This bounds on this term follow from similar computations to the previous term, but we first apply \pref{thormander} to get this term into a form closer to the previous term. After applying this inequality, the remaining computations proceed in a similar fashion. Additionally, we note that by interpolating the estimates in \eref{Qhombound}, we see that 
\begin{equation}
\label{e.fracQhom}
\norm{\nabla^2 \Qhom(s,\cdot)}^2_{H^{-\sfrac{1}{6}}(\Rd)} \leq s^{-\frac{d}{2} - \frac{11}{6}} \,.
\end{equation}
Starting with the interval $[t_0, 4t_0]$, we can use \pref{thormander}, \eref{gradnasharonson}, and \eref{fracQhom} and repeat the same computation as in the previous subsection to see that
\begin{align*}
&\dashint_{\mc{Q}_T}\inp*{\int_{t_0}^{4t_0} (1-\zeta(s))\int_{\Rd} \inprod{P(t-s, x,v,z,w)}_\gamma\ahom:\nabla^2_x \Qhom\,dzds}^2\\
&\quad \leq C\dashint_{\mc{Q}_T}  \norm{P(t-s, x,v,\cdot)}_{L^2((t_0, 4t_0);H^{\sfrac{1}{6}}(\Rd; H^{-1}_\gamma))}^2 \int_{t_0}^{4t_0}s^{-\frac{d}{2} - \frac{11}{6}}\,ds \\
&\quad\leq CT^{-\frac{d}{2}}\,.
\end{align*}
Similarly, moving on to the interval $[t-2t_0, t]$, we see that
\begin{align*}
&\dashint_{\mc{Q}_T}\inp*{\int_{t-2t_0}^{t} (1-\zeta(s))\int_{\Rd} \inprod{P(t-s, x,v,z,w)}_\gamma\ahom:\nabla^2\Qhom(s,z)\,dzds}^2 \\
&\quad\leq C\dashint_{\mc{Q}_T}  \norm{P(t-s, x,v,\cdot)}_{L^2((t-2t_0, t-t_0);H^{\sfrac{1}{6}}(\Rd; H^{-1}_\gamma))}^2 \int_{t-2t_0}^{t-t_0}s^{-\frac{d}{2} - \frac{11}{6}}\,ds  \\
&\quad\leq C T^{-1} t_0^{-\frac{d}{2}+1}\,.
\end{align*}

\subsubsection*{Term 4 bound}
Here, we are looking to bound the following term
$$\dashint_{\mc{Q}_T}\inp*{\int_{\Rd\times\Rd} P(t-t_0, x,v,z,w)(P(t_0, z,w, 0, 0) - \inprod{P(t_0, z,w,0,0)}_\gamma)\,dm(z,w)}^2\,dt \,.$$
First, we notice that, using Fubini's theorem
\begin{align*}
&\dashint_{\mc{Q}_T}\inp*{\int_{\Rd\times\Rd} P(t-t_0, x,v,z,w)(P(t_0, z,w, 0, 0) - \inprod{P(t_0, z,w,0,0)}_\gamma)\,dm(z,w)}^2 \,dt \\
&= \!\dashint_{\mc{Q}_T}\inp*{\int_{\Rd\times\Rd}P(t_0, z,w, 0, 0)( P(t-t_0, x,v,z,w) - \inprod{ P(t-t_0, x,v,z,w)}_\gamma)\,dm(z,w)}^2 \,dt\,.
\end{align*}
Here, we use the Gaussian Poincar\'e inequality as well as \pref{nasharonson} and \eref{gradnasharonson} to see that
\begin{align*}
&\dashint_{\mc{Q}_T}\inp*{\int_{\Rd\times\Rd}P(t_0, z,w, 0, 0)( P(t-t_0, x,v,z,w) - \inprod{ P(t-t_0, x,v,z,w)}_\gamma)\,dm(z,w)}^2 \,dt \\
&\quad\leq \dashint_{\mc{Q}_T}\norm{ P(t_0, \cdot, 0,0)}^2_{L^2(\Rd; L^2_\gamma)}\norm{\nabla_v P(t-t_0, x,v,\cdot)}^2_{L^2(\Rd; L^2_\gamma)} \,dt \\
&\quad\leq Ct_0^{-\frac{d}{2}}\sum_{z\in (\sqrt{T})\Zd} \int_{B_{\sqrt{T}}}T^{\frac{d}{2}}\norm{\nabla_v P(t-t_0, x,v,\cdot)}^2_{\uL^2((\frac{T}{2}, T)\times B_{\sqrt{T}}(z); L^2_\gamma)} \, \\
&\quad\leq C t_0^{-\frac{d}{2}} \sum_{z\in (\sqrt{T})\Zd} T^{\frac{d}{2}-1} \int_{B_{\sqrt{T}}}\Gamma_c(T,x-z)^2 \\
&\quad\leq Ct_0^{-\frac{d}{2}}T^{-1}\,.
\end{align*}

\subsubsection*{Term 5 and beyond bound}
Finally, the bounds on all the remaining terms proceed in a similar manner. To begin, we bound the terms of the form
\begin{align*}
\int_{t_0}^t\int_{\Rd\times\Rd} P(t-s, x,v,z,w) \phi_i (\partial_t(\zeta(s)\partial_{x_i}\Qhom(s,z,w)))\,dm(z,w)ds\,.
\end{align*}
The remaining three terms will then follow from almost identical computations. 

First, we use the estimates proven in \eref{Qhombound} as well as the definition of $\zeta$ to see that
\begin{align*}
\abs{\phi_e\partial_t( \zeta \partial_e \Qhom)} &\leq \abs{\zeta' \phi_e \partial_e\Qhom} + \abs{\zeta \phi_e \partial_t\partial_e\Qhom} \\
&\leq C( t_0^{-1} \abs{\phi_e} s^{-\frac{1}{2}} \Gamma_c  + \abs{\phi_e}s^{-\frac{3}{2}}\Gamma_c)\,.
\end{align*}
Using this bound, we apply H\"older and \pref{nasharonson} as in term 1 to see that
\begin{align*}
\int_{t_0}^t \int_{\Rd} &P(t-s,x,z)\abs{\phi_e\partial_t( \zeta \partial_e \Qhom)(s,z)}\,dzds \\
&\leq  \int_{t_0}^{t-t_0} s^{-\frac{3}{2}} \inp*{\int_{\Rd} \frac{P(t-s,x,z)^2}{\Gamma_c(t-s,x-z)}\,dz}^\frac{1}{2} \inp*{\int_{\Rd} s^3 \abs{\phi_e\partial_t( \zeta \partial_e \Qhom)}^2 \Gamma_c(t-s,x-z)\,dz}^\frac{1}{2} \\
&\leq C\int_{t_0}^{t-t_0} s^{-\frac{3}{2}}\inp*{\int_{\Rd} \abs{\phi_e}^2 \Gamma_c(s,z)^2\Gamma_c(t-s,x-z)\,dz}^\frac{1}{2} \\
&\leq C\inp*{\int_{t_0}^{t-t_0} s^{-\frac{3}{2}} \inp*{\frac{t}{t-s}}^\frac{d}{4} \,ds}\Gamma_c(t,x) \\
&\leq C\inp*{\frac{t}{t_0}}^\frac{d}{4}t_0^{-\frac{1}{2}}\log \frac{t}{t_0}\Gamma_c(t,x)\,
\end{align*}
where the last two steps follow exactly as they did in the bound on the first term. Finally, to conclude the bound on this term, we see that
$$\dashint_{\mc{Q}_T} \inp*{\frac{t}{t_0}}^\frac{d}{2}t_0^{-1}\inp*{\log \frac{t}{t_0}}^2\Gamma_c(t,x)^2 \leq Ct_0^{-\frac{d}{2}-1} \inp*{\log\frac{T}{t_0}}^2\,. $$
To bound the second to last term, we can repeat the above computation to see that
\begin{align*}
\dashint_{\mc{Q}_T}&\inp*{\int_{t_0}^t \zeta(s)\int_{\Rd\times\Rd} P(t-s, x,v,z,w) (\psi_{ij}\partial_t(\zeta(s)\partial_{x_i}\partial_{x_j}\Qhom)(s,z,w))\,dm(z,w)ds}^2 \\
&\leq Ct_0^{-\frac{d}{2}-1} \inp*{\log\frac{T}{t_0}}^2\,.
\end{align*}
Finally, to bound the last term in \eref{duhamel}, we can again repeat the same computation, this time using \lref{orlicz} as well, to see that
\begin{align*}
\dashint_{\mc{Q}_T}&\inp*{\int_{t_0}^t \zeta(s)\int_{\Rd\times\Rd} P(t-s, x,v,z,w) (v_k\psi_{ij} \partial_{x_i}\partial_{x_j}\partial_{x_k} \Qhom)(s,z,w)\,dm(z,w)ds}^2 \\
&\leq Ct_0^{-\frac{d}{2}-1} \inp*{\log\frac{T}{t_0}}^2\,.
\end{align*}
Combining the bounds on all of these terms with our estimate of $\Qhom-\Phom$ yields the desired estimate stated in \tref{perhypohom}.

\section{Large Scale Regularity}
This section is devoted to giving the proof of \tref{largescale}. To begin, we prove a bound on the iterated finite differences of a solution to \eref{hypoeq}, which will give us control of the higher scale effective behavior of the function.
\begin{lem}
\label{l.fdbound}
There exists $C(d,\Lambda)<\infty$ such that, for every $m\in\N$, $R\geq m+2$, and $f\in H_\hyp^1(Q_R)$ a solution of 
$$ -\nabla_v\cdot\a\nabla_v f +v\cdot \a\nabla_v f -v\cdot\nabla_x f + \nabla H \cdot \nabla_v f = 0 \quad \text{ in } Q_R\,,$$
we have the estimate
$$\big\|D^m \hat{f}\;\big\|_{L^\infty(Q_{\sfrac{R}{2}})} \leq \inp*{\frac{Cm}{R}}^m \norm{f}_{\uL^2(Q_R; L^2_\gamma)}\,.$$
\end{lem}
\begin{proof}
To begin, we show that, for $i\in\{1,\ldots, d\}$ and $1\leq r \leq R-2$,
\begin{equation}
\label{e.D1f}
\norm{D_i f}^2_{L^2(Q_r; L^2_\gamma)} \leq \frac{C}{(R-r)^2} \norm{f}_{L^2(Q_R; L^2_\gamma)}^2\,.
\end{equation}
Taking $\xi$ to be a smooth cutoff function such that $\1_{Q_r} \leq \xi \leq \1_{Q_{r+2}}$ and $\abs{\nabla \xi} \leq C$, then applying \pref{hypoelliptic.hormander} and \pref{hypoelliptic.poincare}
\begin{align*}
\norm{D_i f}^2_{L^2(Q_r; L^2_\gamma)} &= \int_{Q_r\times \Rd} (f(x+e_i, v) - f(x,v))^2\,dxd\gamma(v) \\
&\leq \norm{f}_{Q^{\sfrac{1}{3}}_{\nabla_x}(Q_r)}^2 \\
&\leq \norm{\xi f}_{Q^{\sfrac{1}{3}}_{\nabla_x}(\Rd)}^2 \\
&\leq C\norm{\xi f}_{H^1_\hyp(\Rd)}^2 \\
&\leq C\snorm{f}_{H^1_\hyp(Q_{r+2})}^2 \\
&\leq \frac{C}{(R-r)^2} \norm{f}_{L^2(Q_R; L^2_\gamma)}^2\,,
\end{align*}
where the last line follows by applying \lref{caccioppoli}. This proves \eref{D1f}.

Now we iterate the estimate \eref{D1f} to get a bound on the higher order partial difference:
\begin{equation}
\label{e.Dmf}
\norm{D^m f}_{L^2(Q_{\sfrac{3R}{4}}; L^2_\gamma)} \leq \inp*{\frac{Cm}{R}}^m \norm{f}_{L^2(Q_R; L^2_\gamma)}\,.
\end{equation}
As $\nabla H$ is $\Zd$-periodic, $D^n f$ is also a solution for any $n\leq m$. Taking $Q^n = Q_{\inp*{\frac{3}{4} + \frac{n}{4m}} R}$, we can apply \eref{D1f} to see that
$$\norm{D^{n+1}f}_{L^2(Q^{n+1}, L^2_\gamma)} \leq \frac{Cm}{R}\norm{D^nf}_{L^2(Q^n; L^2_\gamma)}\,.$$
Iterating from $n=m$ then gives \eref{Dmf}.

Finally, using this preliminary bound, we will show 
$$\big\|D^m \hat{f}\;\big\|_{L^\infty(Q_{\sfrac{R}{2}})} \leq \inp*{\frac{Cm}{R}}^m \norm{f}_{\uL^2(Q_R; L^2_\gamma)}\,.$$
Applying a discrete Sobolev embedding theorem, we see that
$$\big\|D^m \hat{f} \;\big\|_{L^\infty(Q_{\sfrac{R}{2}})} \leq \big\|D^m \hat{f}\;\big\|_{\uL^2(Q_{\sfrac{R}{2}})} + C R^{\ceil{\frac{d+1}{2}}} \big\|D^{m+\ceil{\frac{d+1}{2}}} \hat{f}\;\big\|_{\uL^2(Q_{\sfrac{R}{2}})}\,.$$
Applying Jensen and using $\eref{Dmf}$, we see that
\begin{align*}
\big\|D^m \hat{f}\;\big\|_{\uL^2(Q_{\sfrac{R}{2}})} &\leq C\norm{D^m f}_{\uL^2(Q_{\sfrac{3R}{4}}; L^2_\gamma)} \\
&\leq \inp*{\frac{Cm}{R}}^m \norm{f}_{\uL^2(Q_R; L^2_\gamma)}\,,
\end{align*}
and
\begin{align*}
R^{\ceil{\frac{d+1}{2}}} \big\|D^{m+\ceil{\frac{d+1}{2}}} \hat{f}\;\big\|_{\uL^2(Q_{\sfrac{R}{2}})} 
&\leq CR^{\ceil{\frac{d+1}{2}}} \norm{D^{m+\ceil{\frac{d+1}{2}}} f}_{\uL^2(Q_{\sfrac{3R}{4}}; L^2_\gamma)} \\
&\leq C m^{\ceil{\frac{d+1}{2}}} \inp*{\frac{Cm}{R}}^m \norm{f}_{\uL^2(Q_R; L^2_\gamma)} \\
&\leq \inp*{\frac{Cm}{R}}^m \norm{f}_{\uL^2(Q_R; L^2_\gamma)}\,,
\end{align*}
where we potentially enlarge $C$ in the final line. Combining these estimates allows us to conclude that
$$\big\|D^m \hat{f}\;\big\|_{L^\infty(Q_{\sfrac{R}{2}})} \leq \inp*{\frac{Cm}{R}}^m \norm{f}_{\uL^2(Q_R; L^2_\gamma)}\,,$$
as required.
\end{proof}

Next, we prove a version of the Poincar\'e inequality (\pref{hypoelliptic.poincare}), which will allow us to relate the above bound to control over the behavior of solutions at lower scales.
\begin{lem}
\label{l.poincarescales}
Let $r\in \N$ with $r\geq 1$, $f \in H_\hyp^1(Q_r)$, and $\g$ such that $\nabla^*_v\g = v\cdot\nabla_x f$ . Then, there exists $C(d)<\infty$ such that
$$\norm{f}_{\uL^2(Q_r; L^2_\gamma)} \leq 
\inp*{ \dashsum_{z\in \Z^d\cap Q_r} \big|\hat{f}(z) \big|^2}^\frac{1}{2} 
+ C(\norm{\nabla_v f}_{\uL^2(Q_r)} + \norm{\g}_{\uL^2(Q_r)})
\,.$$
\end{lem}
\begin{proof}
Applying the triangle inequality and \pref{hypoelliptic.poincare}, we see that
\begin{align*}
\norm{f}_{\uL^2(Q_r; L^2_\gamma)} &\leq \inp*{\dashsum_{z\in \Zd\cap Q_r} \big|\hat{f}(z) \big|^2 + \big\|f-\hat{f}(z) \big\|^2_{\uL^2(z+Q_1; L^2_\gamma)}}^\frac{1}{2} \\
&\leq \inp*{ \dashsum_{z\in \Z^d\cap Q_r} \big|\hat{f}(z) \big|^2}^\frac{1}{2} + C\inp*{\dashsum_{z\in \Zd\cap Q_r} \snorm{f}^2_{H^1_\hyp(z+Q_1)}}^\frac{1}{2} \\
&\leq \inp*{ \dashsum_{z\in \Z^d\cap Q_r} \big|\hat{f}(z) \big|^2}^\frac{1}{2} + C(\norm{\nabla_v f}_{\uL^2(Q_r)} + \norm{\g}_{\uL^2(Q_r)})\,.
\end{align*}
\end{proof}

These two lemmas give us all the material necessary to prove a version of \tref{largescale}, where we take $m$ to be $0$ or $1$. While additional difficulties arise in the case of larger $m$, many of the basic ideas remain consistent. As such, we first prove this restricted result before addressing the issues faced for higher order $m$.
\begin{lem}[Large-scale $C^{0,1}$ and $C^{1,1}$ Estimates]
\label{l.loworderlargescale}
There exists $C(d)<\infty$ such that, for every $R\in[2,\infty)$ and $f\in H_\hyp^1(Q_R)$ a solution of
$$-\nabla_v\cdot\a\nabla_v f +v\cdot \a\nabla_v f -v\cdot\nabla_x f + \nabla H \cdot \nabla_v f = 0 \quad \text{ in } Q_R\,, $$
we have, for every $r\in[1,R]$
\begin{equation}
\label{e.zerocase}
\norm{f- \hat{f}(0)}_{\uL^2(Q_r; L^2_\gamma)} \leq \frac{Cr}{R} \norm{f}_{\uL^2(Q_R; L^2_\gamma)}\,,
\end{equation}
and, with $\psi\in \A_1$ such that $\hat{\psi}(0) = \hat{f}(0)$ and $D\hat{\psi}(0) = D\hat{f}(0)$,
\begin{equation}
\label{e.onecase}
\norm{f-\psi}_{\uL^2(Q_r; L^2_\gamma)} \leq \frac{Cr^2}{R^2} \norm{f}_{\uL^2(Q_R; L^2_\gamma)}\,.
\end{equation}
\end{lem}
\begin{proof}
Using the triangle inequality and applying \lref{fdbound}, we see that, for $r\leq \frac{R}{2}$
$$\sup_{z\in Q_r}\abs{\hat{f}(z)-\hat{f}(0)} \leq Cr \sup_{z\in Q_{r+1}} \abs{D\hat{f}(z)} \leq \frac{Cr}{R}\norm{f}_{\uL^2(Q_R; L^2_\gamma)}\,.$$
So, using \lref{poincarescales} and applying this bound, we see that
$$\norm{f- \hat{f}(0)}_{\uL^2(Q_r; L^2_\gamma)} \leq \frac{Cr}{R}\norm{f}_{\uL^2(Q_R; L^2_\gamma)} + C\snorm{f}_{\uH^1_\hyp(Q_r)}\,.$$
Applying \lref{caccioppoli}, we see that 
$$\snorm{f}_{H^1_\hyp(Q_r)} \leq \frac{C}{s} \norm{f-\hat{f}(0)}_{L^2(Q_{r+s}; L^2_\gamma)}\,,$$
where now we take $r\leq \frac{1}{2} R \wedge (R-s)$.
Choosing $s$ sufficiently large, so that $s\geq 2C$, we can conclude that
$$\norm{f- \hat{f}(0)}_{\uL^2(Q_r; L^2_\gamma)} \leq \frac{Cr}{R}\norm{f}_{\uL^2(Q_R; L^2_\gamma)} + \frac{1}{2} \norm{f-\hat{f}(0)}_{\uL^2(Q_{r+s}; L^2_\gamma)}\,.$$
Iterating this inequality then yields the desired result, for $r\in [1, \frac{1}{2} R \wedge (R-s)]$. Starting from $r= \frac{1}{2} R\wedge (R-s)$ and iterating down the scales gives the desired result for $r\in [1, \frac{1}{2} R\wedge (R-s)]$, provided $R \geq s+1$. However, since $s$ only depends on the dimension, by taking a large value of $C$ in \eref{zerocase}, we can remove this restriction, giving the desired result.

Now, to prove \eref{onecase}, we essentially follow the same argument. In this case, note that since $D^2\psi = 0$ and $\big(\hat{f} - \hat{\psi}\;\big)(0) = 0 = D\big(\hat{f} - \hat{\psi}\;\big)(0)$, 
$$\sup_{z\in Q_r} \abs{\big(\hat{f} - \hat{\psi}\;\big)(z)} \leq Cr^2 \sup_{z\in Q_{r+1}} \abs{D^2\hat{f}(z)} \leq \frac{Cr^2}{R^2}\norm{f}_{\uL^2(Q_R; L^2_\gamma)}\,,$$
where the last inequality follows from \lref{fdbound}. As in the previous version, applying Caccioppoli and \lref{poincarescales} with $f-\psi$ in place of $f-\hat{f}(0)$, we get that
$$\norm{f- \psi}_{\uL^2(Q_r; L^2_\gamma)} \leq \frac{Cr^2}{R^2}\norm{f}_{\uL^2(Q_R; L^2_\gamma)} + \frac{1}{2} \norm{f-\psi}_{L^2(Q_{r+s}; L^2_\gamma)}\,.$$
From here, we finish the proof as above.
\end{proof}

While this approach works perfectly fine when $m=0,1$, when $m\geq 2$, we run into some difficulties when trying to follow the same argument. In particular, we can still find some $\psi\in \A_m$ such that $D^k \hat{\psi} (0) = D^k \hat{f}(0)$ for all $1\leq k\leq m$. However, $p := L\psi$ will no longer necessarily vanish, and thus, in order to apply any bound like \lref{caccioppoli}, we must now also bound this polynomial. As $p$ is a polynomial, we make use of several properties of polynomials as well as properties of the heat kernel to bound this error.
For rest of this section, we define
$$\Gamma(t,x) = \Gamma_1(t,x) = (4\pi t)^{-\frac{d}{2}} \exp\inp*{-\frac{\abs{x}^2}{4t}}\,.$$
To begin, we prove a new version of \lref{poincarescales}, now taking into account the heat kernel.
\begin{lem}
\label{l.heatkernelpoincare}
Let $r\in\N$ and $f\in H^1_\hyp(Q_r)$. Then, there exists a constant $C(d)<\infty$ such that, for every $y\in\Rd$ and $t\in[1,\infty)$, 
$$\int_{Q_r\times\Rd} f^2 \Gamma(t, x-y) \,dm \leq \int_{Q_r} \big|\hat{f}\;\big|^2\Gamma(t, x-y)\,dx + C\int_{Q_r\times \Rd} (\abs{\nabla_v f}^2 + \abs{\g}^2) \Gamma(t+1, x-y)\,dm\,.$$
\end{lem}
\begin{proof}
We first notice that for $t\geq 1$ and $x,x',y\in\Rd$, where $\abs{x-x'}\leq \sqrt{d}$, we can compute that
\begin{equation}
\label{e.heatkernelbound}
\frac{\Gamma(t, x'-y)}{\Gamma(t+1, x-y)} = \inp*{1+\frac{1}{t}}^{\frac{d}{2}}\exp\inp*{-\frac{\abs{x'-y}^2}{4t(t+1)} + \frac{(x-x')\cdot(x+x'-2y)}{4(t+1)}} \leq C(d)\,.
\end{equation}
Therefore, combining this bound with \pref{hypoelliptic.poincare}, we have that
\begin{align*}
\int_{(z+Q_1)\times \Rd} \big| f-\hat{f}\;\big |^2 \Gamma(t, x-y)\,dm &\leq C\sup_{x\in z+Q_1}\Gamma(t, x-y) \int_{(z+Q_1)\times\Rd} \abs{\nabla_v f}^2 + \abs{\g}^2\,dm \\
&\leq C \int_{(z+Q_1)\times\Rd} (\abs{\nabla_v f}^2 + \abs{\g}^2)\Gamma(t+1, x-y)\,dm\,.
\end{align*}
Summing over $z\in Q_r$ then gives that
$$\int_{Q_r\times \Rd} \abs{f-\hat{f}\;}^2 \Gamma(t, x-y)\,dm \leq C \int_{Q_r\times\Rd} (\abs{\nabla_v f}^2 + \abs{\g}^2)\Gamma(t+1, x-y)\,dm\,.$$
Applying the reverse triangle inequality to the left hand side then yields the desired result.
\end{proof}

The following technical lemma gives us the necessary bounds on polynomials with respect to the heat kernel we need in order to adapt \lref{caccioppoli} and complete the proof of \tref{largescale}. The first two results are proved in \cite[Lemma 3.6]{AKS}, but we reproduce them here for the sake of completeness.
\begin{lem}[Polynomial Estimates]
\label{l.polyest}
For every $y\in\Rd$, $t > 0$, $m\in\N$, $p\in\P_m$, $r\geq \sqrt{64mt}$, and  $R_2\geq R_1 \geq Cm$, we have 
\begin{align}
\int_{\Rd} \inp*{\frac{\nabla(p\Gamma(t, x-y))}{\Gamma(t, x-y)}}^2 \Gamma(t, x-y) &\leq \frac{2(m+d)}{t}\int_{\Rd} p^2 \Gamma(t, x-y)\,, \label{e.heatkernelpolybound}\\
\int_{\Rd\setminus Q_r} p^2\Gamma(t, x-y) &\leq e^{-\sfrac{r^2}{16t}} \int_{\Rd} p^2\Gamma(t, x-y)\,, \label{e.polytailcutoff} \\
\int_{\Rd} \abs{\nabla p}^2 \Gamma(t, x-y) &\leq \frac{Cm}{t} \int_{\Rd} p^2\Gamma(t, x-y) \,,\label{e.polymarkov} \\
\abs{\nabla^n p(0)} &\leq \sum_{k=n}^m C^k \abs{D^k p(0)} \label{e.derivfdbound}\,, \\
\norm{p}_{L^\infty(Q_{R_2})} &\leq \inp*{\frac{CR_2}{R_1}}^m \norm{p}_{\uL^2(Q_{R_1})}\,. \label{e.polyinftybound}
\end{align}
\end{lem}
\begin{proof}\
We begin by proving \eref{heatkernelpolybound}. Up to scaling, it suffices to consider the case where $t = \frac{1}{4}$ and $y=0$, and we let $\Gamma_0(x) := \Gamma(\sfrac{1}{4}, x)$. Furthermore, we let $h_\alpha(x) := (-1)^{\abs{\alpha}} \Gamma_0^{-1} \partial^\alpha \Gamma_0(x)$, for $\alpha$ a multi-index, denote the physicist Hermite polynomials. A quick computation from this definition shows that $\partial_{x_i}(h_\alpha \Gamma_0(x)) = -h_{\alpha+e_i} \Gamma_0(x)$. Repeatedly integrating by parts and using this equation, we see that, for multi-indices $\alpha$ and $\beta$ with $\abs{\beta}\geq \abs{\alpha}$,
$$\int_{\Rd} h_{\alpha}(x)h_{\beta}(x) \Gamma_0(x)\,dx = \int_{\Rd} \partial^\beta h_{\alpha}(x) \Gamma_0(x)\,dx = 2^{\abs{\alpha}} \alpha! \1_{\alpha = \beta}\,,$$
that is the Hermite polynomials are orthogonal with respect to the $L^2$ inner product with weight $\Gamma_0$. Using this, we can represent any $p\in\P_m$ by $p = \sum_{\abs{\alpha} \leq m} c_{\alpha} h_{\alpha}$, where $c_{\alpha} = (2^{\abs{\alpha}} \alpha!)^{-1} \int_{\Rd} p(x) h_\alpha(x) \Gamma_0(x)\,dx.$ Using this representation, we can see that
\begin{align*}
\int_{\Rd} \inp*{\frac{\nabla(p\Gamma_0)}{\Gamma_0}}^2 \Gamma_0 &\leq \sum_{i=1}^d \int_{\Rd} \abs*{\sum_{\abs{\alpha}\leq m} c_\alpha \partial_{x_i}(h_\alpha \Gamma_0)}^2 \Gamma_0^{-1} \\
&= \sum_{i=1}^d \int_{\Rd} \abs*{\sum_{\abs{\alpha}\leq m} c_\alpha h_{\alpha+e_i}}^2 \Gamma_0 \\
&= \sum_{i=1}^d \sum_{\abs{\alpha}\leq m}\int_{\Rd}  c_\alpha^2 h_{\alpha+e_i}^2 \Gamma_0 \\
&= \sum_{i=1}^d \sum_{\abs{\alpha}\leq m}2(\alpha_i + 1)\int_{\Rd}  c_\alpha^2 h_{\alpha}^2 \Gamma_0 \\
&\leq 2(m+d) \int_{\Rd} p^2 \Gamma_0\,,
\end{align*}
which is \eref{heatkernelpolybound}.

Next, we prove \eref{polytailcutoff}. Using the notation of the previous paragraph, we can simplify the recurrence for the Hermite polynomials used above to see that,
$$h_{\alpha+e_i} = 2x_i h_\alpha - 2\alpha_i h_{\alpha - e_i}\,,$$
where we set $h_\beta = 0$ if any of the indices of $\beta$ are negative. Inducting on this formula, we see that, provided $\abs{x}^2\geq 4\abs{\alpha}$, we have that
$$\abs{h_\alpha(x)} \leq \prod_{i=1}^d \abs{4x_i}^{\alpha_i}\,.$$
Furthermore, as $\abs{\alpha}\leq m$ implies that $0\leq \abs{\alpha_i}\leq m$ for all $i\in\{1, \ldots, d\}$, we have the following bound 
$$\abs{\{\alpha\in\Nd: \abs{\alpha}\leq m\}} \leq (m+1)^d \leq (2m)^d\,.$$
Then, for $R\geq \sqrt{4m}$, we can use these bounds to see that
\begin{align*}
\int_{\Rd\setminus Q_R} p^2 \Gamma_0 &= \int_{\Rd\setminus Q_R} \inp*{\sum_{\abs{\alpha}\leq m} c_\alpha h_\alpha}^2 \Gamma_0 \\
&\leq (2m)^d\sum_{\abs{\alpha}\leq m} c_\alpha^2\int_{\Rd\setminus Q_R} h_\alpha^2\Gamma_0 \\
&\leq (2m)^d\sum_{\abs{\alpha}\leq m} 4^{2\abs{\alpha}}c_\alpha^2 \prod_{i=1}^d\int_R^\infty\abs{x_i}^{2\alpha_i} e^{-x_i^2} \,dx_i\,.
\end{align*}
A quick calculus computation shows that, for any $n\in\N$ and $t\geq 0$,
$$t^n e^{-\frac{t}{2}} \leq \inp*{\frac{2n}{e}}^n \leq 2^n n!\,.$$
Then, changing variables and using this inequality, we see that 
\begin{align*}
\int_R^\infty t^{2n} e^{-t^2}\,dt &= \frac{1}{2}\int_{R^2}^\infty s^{n-\frac{1}{2}} e^{-s}\,ds \\
&\leq \frac{2^n n!}{2R} \int_{R^2}^\infty e^{-\frac{s}{2}}\,ds \\
&= \frac{2^n n!}{R} e^{-\sfrac{R^2}{2}}\,. 
\end{align*}
So, continuing where we left off above, this gives that
\begin{align*}
\int_{\Rd\setminus Q_R} p^2\Gamma_0 &\leq (2m)^d16^m\sum_{\abs{\alpha}\leq m}c_\alpha^2 \prod_{i=1}^d\int_R^\infty\abs{x_i}^{2\alpha_i} e^{-x_i^2} \,dx_i \\
&\leq \inp*{\frac{2m}{R}}^d16^m e^{-\sfrac{dR^2}{2}} \sum_{\abs{\alpha}\leq m} 2^{\abs{\alpha}}\alpha! c_\alpha^2\,.
\end{align*}
Using the orthogonality of the Hermite polynomials, we see that 
$$\int_{\Rd} p^2 \Gamma_0 = \sum_{\abs{\alpha}\leq m} 2^{\abs{\alpha}}\alpha! c_\alpha^2\,.$$
Finally, to conclude, we note that for $R\geq \sqrt{16m}$, we can use that $2m \leq e^m$ to see that 
$$\inp*{\frac{2m}{R}}^d16^m e^{-\sfrac{dR^2}{2}} \leq e^{(d+\log16)m - \sfrac{dR^2}{2}} \leq e^{4dm - \sfrac{dR^2}{2}} \leq e^{-\sfrac{R^2}{4}}\,,$$
which completes the proof of \eref{polytailcutoff}.

Now, to prove \eref{polymarkov}, we first note that, using the product rule and bound $\abs{\nabla \Gamma(t, \cdot)}^2$, 
\begin{align*}
\abs{\nabla p}^2 \Gamma(t,\cdot) &\leq \frac{2\abs{\nabla(p\Gamma(t,\cdot))}^2}{\Gamma(t,\cdot)^2} \Gamma(t,\cdot) + \frac{2p^2 \abs{\nabla \Gamma(t,\cdot)}^2}{\Gamma(t,\cdot)^2} \Gamma(t,\cdot) \\
&\leq \frac{2\abs{\nabla(p\Gamma(t,\cdot))}^2}{\Gamma(t,\cdot)^2} \Gamma(t,\cdot) + \frac{p^2 \abs{x}^2}{2t^2} \Gamma(t,\cdot)\,.
\end{align*}
The first term we bound by \eref{heatkernelpolybound} and for the second term, we can use \eref{polytailcutoff} to see that
\begin{align*}
\int_{\Rd} \frac{p^2 \abs{x}^2}{2t^2} \Gamma(t,\cdot) &\leq \int_{Q_{\sqrt{64(m+2)t}}} \frac{\abs{x}^2 p^2}{t^2} \Gamma(t,\cdot) \\
&\leq \frac{64(m+2)t}{t^2} \int_{\Rd} p^2 \Gamma(t,\cdot)\,.
\end{align*}
This gives \eref{polymarkov}.

Continuing, we will now prove \eref{derivfdbound}. We notice that this bound is essentially proved for homogenous polynomials in \eref{derivdiffbound}, so it only remains to extend this result to arbitrary $p\in \P_m$. While inhomogeneous $p$ no longer satisfy the representation formula that we used in proving \eref{derivdiffbound}, we will present a new representation formula that does work in this case and use this to derive the bound. 

We define the Newton polynomials for some $k\in\N$ by 
$$N_k(x) = \frac{1}{k!}\prod_{j=0}^{k-1} (x-j)\,,$$
with $N_0 = 1$. An immediate computation shows that 
$$DN_k(x) = N_{k-1}(x)\,.$$
In particular, defining the multivariate Newton polynomials by 
$$N_\alpha(x_1, \ldots, x_d) = \prod_{i=1}^d N_{\alpha_i}(x_i)\,,$$
we see that for multi-indices $\alpha, \beta$, 
$$D^\beta N_\alpha (x) = N_{\alpha - \beta}\,.$$
Using this property, it follows that any $p\in \P_m$ can be represented as 
$$p(x) = \sum_{i=0}^m \sum_{\abs{\alpha} = i} D^\alpha p(0) N_\alpha(x)\,.$$
So, taking the derivative of this expression, we see that
$$\partial^\beta p(x) = \sum_{i=\abs{\beta}}^m \sum_{\abs{\alpha} = i} D^\alpha p(0) \partial^\beta N_\alpha(x)\,.$$
To complete the proof, it only remains to bound $\partial^\beta N_\alpha(0)$. We claim that $\abs{\partial^\beta N_\alpha(0)} \leq \binom{\alpha}{\beta}.$ From the definition of the multivariate Newton polynomials, it suffices to show this for the single dimensional polynomials. By definition, $N_k(x) = \frac{1}{k}(x-k+1)N_{k-1}(x)$, so we can take the derivative of this to see that 
$$\partial^j N_k(x)= \frac{j}{k} \partial^{j-1} N_{k-1}(x) + \frac{1}{k}(x-k+1)\partial^j N_{k-1}(x)\,.$$
So, inducting on $j$ and $k$, with $j\leq k$, we see that
\begin{align*}
\abs{\partial^j N_k(0)} &\leq \frac{j}{k} \partial^{j-1} \abs{N_{k-1}(0)} + \frac{k-1}{k}\abs{\partial^j N_{k-1}(0)} \\
&\leq \frac{j}{k} \binom{k-1}{j-1} + \frac{k-1}{k}\binom{k-1}{j} \\
&= \binom{k}{j}\inp*{\frac{j^2}{k^2} + \frac{(k-1)(k-j)}{k^2}} \\
&=\binom{k}{j}\frac{k^2 - (j+1)(k-j)}{k^2} \\
&\leq \binom{k}{j}\,.
\end{align*}
So, using this bound and that $\binom{k}{j} \leq C^k$, we get \eref{derivfdbound}.

Finally, to prove \eref{polyinftybound}, we will first use \eref{polymarkov} to show that for $k\in\{0,\ldots, m\}$ and $r>0$,
$$\norm{\nabla^k p}_{\uL^2(Q_r)} \leq C \inp*{\frac{Cm}{r}}^k \norm{p}_{\uL^2(Q_{2r})}\,.$$
Noting that, for $x\in Q_r$
$$\int_{Q_r} \Gamma(\sfrac{r^2}{64m}, x-y) \,dy \geq c\,,$$
for some dimensional constant $c$, we can then apply \eref{polytailcutoff} and \eref{polymarkov} to see that
\begin{align*}
\dashint_{Q_r} \abs{\nabla^k p(x)}^2\,dx &\leq C\dashint_{Q_r} \abs{\nabla^k p(x)}^2 \int_{Q_r} \Gamma(\sfrac{r^2}{64m}, x-y)\,dydx \\
&\leq C\inp*{\frac{Cm}{r}}^{2k} \dashint_{Q_r} \int_{\Rd} p(x)^2 \Gamma(\sfrac{r^2}{64m}, x-y)\,dydx \\
&\leq C\inp*{\frac{Cm}{r}}^{2k} \dashint_{Q_r} \int_{y+Q_r} p(x)^2 \Gamma(\sfrac{r^2}{64m}, x-y)\,dydx \\
&\leq C\inp*{\frac{Cm}{r}}^{2k} \dashint_{Q_{2r}}p(x)^2 \int_{\Rd}\Gamma(\sfrac{r^2}{64m}, x-y)\,dydx \\
&\leq C\inp*{\frac{Cm}{r}}^{2k} \dashint_{Q_{2r}}p(x)^2 \,dx\,,
\end{align*}
which gives the desired bound. Then, applying the Sobolev inequality and using this identity, we see that, for $r\geq Cm$, 
$$\norm{\nabla^k p}_{L^\infty(Q_{\sfrac{r}{2}})} \leq \norm{\nabla^k p}_{\uL^2(Q_{\sfrac{r}{2}})} + Cr^{\ceil{\frac{d+1}{2}}} \norm{\nabla^{k + \ceil{\frac{d+1}{2}}} p}_{\uL^2(Q_{\sfrac{r}{2}})} \leq Cm^{\ceil{\frac{d+1}{2}}}\inp*{\frac{Cm}{r}}^k \norm{p}_{\uL^2(Q_r)}\,.$$
To conclude, we can use that $p(x) = \sum_{\abs{\alpha}\leq m} \frac{1}{\alpha!}\partial^\alpha p(0) x^\alpha$ and apply this inequality to see that, for $R_2\geq R_1 \geq Cm$,
\begin{align*}
\norm{p}_{L^\infty(Q_{R_2})} &\leq \sum_{k=0}^m \inp*{\frac{CR_2}{k+1}}^k \abs{\nabla^k p(0)} \\
&\leq C^m \sum_{k=0}^m \inp*{\frac{CR_2}{k+1}}^k \inp*{\frac{k+1}{R_1}}^k \norm{p}_{\uL^2(Q_{R_1})} \\
&\leq \inp*{\frac{CR_2}{R_1}}^m \norm{p}_{\uL^2(Q_{R_1})}\,,
\end{align*}
which gives the desired bound.
\end{proof}

We now have all the necessary ingredients to prove the necessary adaptation to \lref{caccioppoli}.
\begin{lem}[Polynomial Lemma]
\label{l.polylem}
For each $\delta\in(0,1]$, there exists $C<\infty$ and $c>0$ such that, for every $m\in\N$, $R\geq Cm$, $t\in [Cm, R]$, $y\in Q_{\sfrac{R}{2}}$, $p\in\P_m$ and $f\in H^1_\hyp(Q_R)$ a solution of
\begin{equation}
\label{e.pdepolylemma}
-\nabla_v\cdot\a\nabla_v f + v\cdot\a\nabla_v f - v\cdot\nabla_x f + \nabla H(x)\cdot \nabla_vf  = p \quad \text{ in } Q_R\,,
\end{equation}
we have the estimate
\begin{align*}
 \int_{Q_{\sfrac{3R}{4}}} (\abs{\nabla_v f}^2 + \abs{\g}^2 + p^2) \Gamma(t,x-y) &\leq \delta \int_{Q_{\sfrac{3R}{4}}\times \Rd} f^2 \Gamma(t+C,x-y)\,dxd\gamma \\
 &\quad+ \exp(-cR)\int_{Q_R\times \Rd} f^2\,dxd\gamma\, . \nonumber
\end{align*}
\end{lem}
\begin{proof}
Take $\delta\in(0,1]$ and define $\psi$ to be a smooth cutoff function satisfying 
$$\1_{Q_{\sfrac{3R}{4}}} \leq \psi \leq \1_{Q_R} \quad \text{ and } \quad \abs{\nabla \psi} \leq CR^{-1}. $$
Testing \eref{pdepolylemma} with $f\psi^2\Gamma(t,x-y)$, integrating by parts twice and using the ellipticity of $\a$, we see that
\begin{align*}
\int_{Q_R\times \Rd} &pf\psi^2 \Gamma(t,x-y)\,dm \\
&= \int_{Q_R\times\Rd} (\psi^2 \nabla_v f \cdot\a\nabla_v f - f\psi^2 (v\cdot\nabla_x f - \nabla H(x)\cdot \nabla_vf))\Gamma(t,x-y) \,dm \\
&\geq \int_{Q_R\times\Rd} \psi^2 \abs{\nabla_v f}^2\Gamma(t,x-y) + \frac{1}{2}f^2 v\cdot \nabla_x(\psi^2\Gamma(t,x-y))  \,dm \\
&= \int_{Q_R\times\Rd} \psi^2 \abs{\nabla_v f}^2\Gamma(t,x-y) + \frac{1}{2}f \nabla_v f \cdot \nabla_x(\psi^2\Gamma(t,x-y))  \,dm \\
&\geq \frac{1}{2} \int_{Q_R\times \Rd} \abs{\nabla_v f}^2 \psi^2\Gamma(t,x-y)\,dm \\
&\quad
- \int_{Q_R\times \Rd} f^2\Gamma(t,x-y)\inp*{ \psi^2 \frac{\abs{\nabla \Gamma(t,x-y)}^2}{\Gamma(t,x-y)^2} + \abs{\nabla \psi}^2}\,dm\,.
\end{align*}
Applying Young's inequality to the left hand side above and rearranging, we see that
\begin{align}
\label{e.testfexp}
\int_{Q_R\times \Rd} \abs{\nabla_v f}^2 \psi^2 &\Gamma(t,x-y)\,dm \\
&\leq \int_{Q_R\times \Rd} f^2\inp*{\frac{\delta}{8} \psi^2 + C\psi^2 \frac{\abs{\nabla \Gamma(t,x-y)}^2}{\Gamma(t,x-y)^2} + C\abs{\nabla\psi}^2} \Gamma(t,x-y)\,dm \nonumber\\
&\quad+ \frac{C}{\delta}\int_{Q_R} p^2\psi^2 \Gamma(t,x-y)\,dx\,.\nonumber
\end{align}
Testing \eref{pdepolylemma} instead with $p\psi^2 \Gamma(t,x-y)$ and performing the same computation, we see that
\begin{align}
\label{e.testpexp}
\int_{Q_R} p^2\psi^2 \Gamma(t,x-y)\,dx
&\leq \frac{\delta}{16} \int_{Q_R\times \Rd} \abs{\nabla_v f}^2 \psi^2 \Gamma(t,x-y)\,dm \\
&\quad+ \frac{C}{\delta} \int_{Q_R} \inp*{\frac{\abs{\nabla(p\Gamma(t,x-y))}^2}{\Gamma(t,x-y)^2} 
+ p^2\abs{\nabla\psi}^2} \Gamma(t,x-y)\,dx\,.\nonumber
\end{align}
Combining \eref{testfexp} and \eref{testpexp}, we see that
\begin{align}
\label{e.testingerr}
\int_{Q_R\times \Rd} \inp*{\abs{\nabla_v f}^2 + \tfrac{1}{\delta} p^2}&\psi^2\Gamma(t,x-y)\,dm \\
&\leq \int_{Q_R\times\Rd} f^2\inp*{\frac{\delta}{4} + C\frac{\abs{\nabla \Gamma(t,x-y)}^2}{\Gamma(t,x-y)^2}} \psi^2\Gamma(t,x-y)\,dm \nonumber \\
&\quad+ \frac{C}{\delta^2} \int_{Q_R} \inp*{\frac{\abs{\nabla(p\Gamma(t,x-y))}^2}{\Gamma(t,x-y)^2}}\psi^2 \Gamma(t,x-y)\,d\sigma \nonumber\\
&\quad+ C \int_{Q_R\times \Rd} \inp*{f^2 + \tfrac{1}{\delta^2}p^2}\abs{\nabla\psi}^2 \Gamma(t,x-y)\,dm \nonumber \,.
\end{align}
Now we estimate the terms on the right hand side of \eref{testingerr}

Beginning with the first and third term, we can compute that, for $s\in(0,t]$, 
$$
\frac{\abs{\nabla \Gamma(t,x-y)}^2}{\Gamma(t,x-y)^2} \frac{\Gamma(t,x-y)}{\Gamma(t+s,x-y)} \leq \frac{\abs{x-y}^2}{4t^2} \frac{\Gamma(t,x-y)}{\Gamma(t+s,x-y)} \leq \frac{2}{s} \inp*{\frac{s\abs{x-y}^2}{8t^2} \exp\inp*{-\frac{s\abs{x-y}^2}{8t^2}}} \leq \frac{1}{s}\,.
$$
Thus, we can find a sufficiently large $C$ so that for any $y\in Q_{\sfrac{R}{2}}$ and $t\in[C,R]$, by the above computation, 
$$C\frac{\abs{\nabla \Gamma(t,x-y)}^2}{\Gamma(t,x-y)^2} \frac{\Gamma(t,x-y)}{\Gamma(t+C,x-y)} \leq \frac{\delta}{4}\,.$$
Consequently, we now have that
\begin{align}
\label{e.errbound1}
\int_{Q_{\sfrac{3R}{4}}\times\Rd} f^2\inp*{\frac{\delta}{4} + C\frac{\abs{\nabla \Gamma(t,x-y)}^2}{\Gamma(t,x-y)^2}} &\psi^2\Gamma(t,x-y)\,dm \\
&\leq
\frac{\delta}{2} \int_{Q_{\sfrac{3R}{4}}\times\Rd} f^2 \Gamma(t+C,x-y)\,dm\,. \nonumber
\end{align}
As $y\in Q_{\sfrac{R}{2}}$ and $\nabla \psi$ is supported in $Q_R\setminus Q_{\sfrac{3R}{4}}$, we can use bound on $\Gamma(t,x-y)$ for $x\in Q_R\setminus Q_{\sfrac{3R}{4}}$ to see that
\begin{align}
\label{e.errbound2}
\int_{Q_R\setminus Q_{\sfrac{3R}{4}}\times\Rd} &f^2\inp*{\frac{\delta}{2} 
+ C\frac{\abs{\nabla \Gamma(t,x-y)}^2}{\Gamma(t,x-y)^2}} \psi^2\Gamma(t,x-y)\,dm \\
&\quad + \int_{Q_R\times\Rd} (f^2+\tfrac{1}{\delta^2}p^2)\abs{\nabla\psi}^2 \Gamma(t,x-y)\,dm \nonumber\\
&\leq C\exp\inp*{-\frac{R^2}{Ct}} \int_{Q_R\times\Rd} (f^2+p^2)\,dm\,, \nonumber
\end{align}
which bounds the first and third term.
For the second term, we can apply \lref{polyest} to see that
\begin{align}
\label{e.errbound3}
\int_{\Rd} \frac{\abs{\nabla (p\Gamma(t,x-y))}^2}{\Gamma(t,x-y)^2} \Gamma(t,x-y)\,dx &\leq \frac{4dm}{t} \int_{\Rd} p^2\Gamma(t,x-y)\,dx\\
& \leq \frac{8dm}{t} \int_{Q_{\sfrac{3R}{4}}} p^2\Gamma(t,x-y)\,dx\,.\nonumber
\end{align}
Combining \eref{errbound1}, \eref{errbound2}, and \eref{errbound3} with \eref{testingerr}, we see that
\begin{align*}
\int_{Q_{\sfrac{3R}{4}}\times\Rd} (\abs{\nabla_v f}^2 + \tfrac{1}{\delta}p^2) \Gamma(t,x-y)\,dm 
&\leq \frac{\delta}{2} \int_{Q_{\sfrac{3R}{4}}\times\Rd} f^2 \Gamma(t+C,x-y) \,dm \\
&\quad+ C\exp\inp*{-\frac{R^2}{Ct}} \int_{Q_R\times\Rd} (f^2+p^2)\,dm \\
 &\quad+\frac{Cm}{\delta^2t}\int_{Q_{\sfrac{3R}{4}}} p^2 \Gamma(t,x-y)\,dx\,.
\end{align*}
If $t\in [2C\delta^{-1}m, R]$, then we can combine the final term on the right hand side with the left hand side, which yields
\begin{align}
\label{e.expcaccio}
\int_{Q_{\sfrac{3R}{4}}\times\Rd} (&\abs{\nabla_v f}^2 + p^2) \Gamma(t,x-y)\,dm \\
&\leq \delta \int_{Q_{\sfrac{3R}{4}}\times\Rd} f^2 \Gamma(t+C,x-y)\,dm + C\exp\inp*{-cR} \int_{Q_R\times\Rd} (f^2+p^2)\,dm\,. \nonumber
\end{align}
Now it only remains to estimate the final term on the right hand side of this inequality. Testing \eref{pdepolylemma} with $f \psi^2$ and $p \psi^2$ and performing a similar computation to \eref{testfexp} and \eref{testpexp} gives, for $\theta\in (0,\infty)$
\begin{align}
\label{e.testf}
\int_{Q_R\times\Rd} \abs{\nabla_v f}^2 \psi^2\,dm &\leq \frac{\theta R^2}{2} \int_{Q_R} p^2 \psi^2\,d\sigma \\
&\quad+ C \int_{Q_R\times\Rd} (\theta^{-1}R^{-2} + \abs{\nabla \psi}^2) f^2\,dm\,, \nonumber \\
\label{e.testp}
\int_{Q_R} p^2 \psi^2\,d\sigma &\leq \theta^{-1}R^{-2} \int_{Q_R\times\Rd} \abs{\nabla_v f}^2 \psi^2\,dm \\
&\quad+ C\theta R^2 \int_{Q_R} \abs{\nabla p}^2\psi^2\,dx + C\theta \int_{Q_R} \abs{\nabla \psi}^2 p^2\,dx\,. \nonumber
\end{align}
Using the definition of $\psi$ and combining \eref{testf} and \eref{testp} we see that
\begin{equation}
\label{e.finalpbound}
\int_{Q_{\sfrac{3R}{4}}} p^2\,dx \leq C\theta R^2\int_{Q_R} \abs{\nabla p}^2\,dx + C\theta \int_{Q_R} p^2\,dx + \frac{C}{\theta^2 R^4}\int_{Q_R\times\Rd} f^2\,dm\,.
\end{equation}
Combining \eref{polytailcutoff}, \eref{polymarkov}, and \eref{polyinftybound} with \eref{finalpbound} and taking  $\theta = C^{-m}$, we see that
\begin{equation}
\label{e.polylinftybound}
\norm{p}_{L^\infty(Q_R)}^2 \leq C^m \int_{Q_{\sfrac{3R}{4}}} p^2\,dx \leq \frac{C^m}{R^4} \int_{Q_R\times\Rd} f^2\,dm\,.
\end{equation}
Thus, returning to the term we want to bound, if $R\geq Cm$, we get that
\begin{align*}
\exp(-cR) \int_{Q_R\times\Rd} (f^2+p^2)\,dm &\leq (1+C^m)\exp(-cR)\int_{Q_R\times\Rd} f^2\,dm \\
&\leq \exp\inp*{-\tfrac{1}{2} cR} \int_{Q_R\times\Rd} f^2\,dm\,,
\end{align*}
which completes the first part of the proof.

Finally, it only remains to show 
$$  \int_{Q_{\sfrac{3R}{4}}} \abs{\g}^2 \Gamma(t,x-y) \leq \delta \int_{Q_{\sfrac{3R}{4}}\times \Rd} f^2 \Gamma(t+C,x-y)\,dxd\gamma + \exp(-cR)\int_{Q_R\times \Rd} f^2\,dxd\gamma\,.$$
Testing \eref{pdepolylemma} with $\phi\Gamma(t,x-y)$, where $\phi\in L^2(Q_{\sfrac{3R}{4}};H^1_\gamma)$ gives that 
$$\int_{Q_{\sfrac{3R}{4}}\times \Rd} \phi Bf \Gamma(t,x-y)\,dxd\gamma = \int_{Q_{\sfrac{3R}{4}}\times \Rd}( \nabla_v\phi \cdot\a\nabla_v f +p\phi) \Gamma(t,x-y)\,dxd\gamma \,.$$
Therefore, using \eref{expcaccio} and taking the supremum over $\phi$ such that $\norm{\phi}_{L^2(Q_{\sfrac{3R}{4}};H^1_\gamma)}\leq 1$ yields
$$\norm{Bf\Gamma(t,x-y)}^2_{L^2(Q_{\sfrac{3R}{4}}; H^{-1}_\gamma)} \leq \delta \int_{Q_{\sfrac{3R}{4}}\times \Rd} f^2 \Gamma(t+C,x-y)\,dxd\gamma + \exp(-cR)\int_{Q_R\times \Rd} f^2\,dxd\gamma\,,$$
which completes the proof.
\end{proof}
Now we can move on to extending \lref{loworderlargescale} to higher orders.
\begin{proof}[Proof of \tref{largescale}]
As we already proved the result for $m=0,1$ in \lref{loworderlargescale}, we only need to consider the case where $m\geq 2$. As before, we find $\psi\in\A_m$ such that, for $k\in \{0,\ldots, m\}$
\begin{equation*}
D^k \hat{\psi}\,(0) = D^k \hat{f}\,(0)
\end{equation*}
and 
\begin{align}
 \label{e.lscorrectorbound}
\norm{\psi}_{{\uL^2(Q_r)}} &\leq \sum_{k=0}^m \inp*{\frac{Cr}{k+1}}^k \abs{D^k\hat{f}\,(0)} \\
&\leq C^m \norm{f}_{\uL^2(Q_r)}\,. \nonumber
\end{align}
Denoting $g = f-\psi$, we can use that $D^{m+1}\psi = 0$ and apply \lref{fdbound} to get that
\begin{equation}
\label{e.vavgbound}
\sup_{z\in Q_r} \abs{\,\hat{g}\,(z)} \leq \frac{Cr^{m+1}}{(m+1)!} \sup_{z\in Q_{r+m+1}} \abs{D^{m+1}\hat{f}\,(z)} \leq \inp*{\frac{Cr}{R}}^{m+1} \norm{f}_{\uL^2(Q_R; L^2_\gamma)}\,.
\end{equation}

Now, applying \lref{heatkernelpoincare}, for every $r\in [Cm, R]$ and $y\in Q_{\sfrac{R}{2}}$, we see that
\begin{align}
\label{e.vhkpoincare}
&\int_{Q_{\sfrac{3R}{4}}\times \Rd} g^2 \Gamma(r,x-y)\,dm \\
&\qquad\leq C \int_{Q_{\sfrac{3R}{4}}} {\hat{g}\,}^2 \Gamma(r,x-y)\,dx + C \int_{Q_{\sfrac{3R}{4}}\times\Rd} (\abs{\nabla_v g}^2 + \abs{\h}^2)\Gamma(r+1,x-y)\,dm\,,\nonumber
\end{align}
where $\h$ satisfies $\nabla_v^*\h = v\cdot\nabla_x g$. Using \eref{vavgbound}, we can bound the first term on the right hand side of \eref{vhkpoincare} by
\begin{align*}
\int_{Q_{\sfrac{3R}{4}}} {\hat{g}\,}^2\Gamma(r,x-y) &\leq \int_{Q_{\sfrac{7R}{8}}} \hat{g}\,(\cdot-y)^2\Gamma(r,x) \\
&= \int_0^{2r}\int_{\partial Q_s} \hat{g}\,(\cdot -y)^2 \Gamma(r,x)\,ds + \int_{2r}^{\sfrac{7R}{8}}\int_{\partial Q_s} \hat{g}\,(\cdot -y)^2 \Gamma(r,x)\,ds \\
&\leq \sup_{Q_{3r}} {\hat{g}\,}^2 \int_{Q_r}\Gamma(r,x) + \int_{2r}^{\sfrac{7R}{8}}\sup_{Q_{s+r}} {\hat{g}\,}^2 e^{-\frac{cs^2}{r}}\,ds \\
&\leq \norm{f}_{\uL^2(Q_R; L^2_\gamma)}^2\inp*{\inp*{\frac{Cr}{R}}^{2(m+1)} + \int_{2r}^{\sfrac{7R}{8}}\inp*{\frac{Cs}{R}}^{2(m+1)}  e^{-\frac{cs^2}{r}}\,ds }\\
&\leq \inp*{\frac{Cr}{R}}^{2(m+1)} \norm{f}_{\uL^2(Q_R;L^2_\gamma)}^2\,.
\end{align*}
Expanding on the last line in the above computation, we bound the given exponential integral by changing variables and computing, for $r\geq Cm$,
\begin{align*}
\int_{2r}^{\sfrac{7R}{8}}\inp*{\frac{Cs}{R}}^{2(m+1)}  e^{-\frac{cs^2}{r}}\,ds &\leq \inp*{\frac{Cr}{R}}^{2(m+1)} \int_2^\infty s^{2(m+1)} re^{-crs^2}\,ds \\
 &\leq \inp*{\frac{Cr}{R}}^{2(m+1)} \int_2^\infty s^{2(m+1)} me^{-ms^2}\,ds \\
 &\leq \inp*{\frac{Cr}{R}}^{2(m+1)} \int_m^\infty \frac{1}{2}\inp*{\frac{s}{m}}^{m+\frac{1}{2}}e^{-s}\,ds \\
 &\leq \inp*{\frac{Cr}{R}}^{2(m+1)}\,.
 \end{align*}
To bound the second term, we take $\delta>0$ and apply \lref{polylem}, with $L g = -L \psi \in \P_{m-2}$, to obtain 
\begin{align*}
\int_{Q_{\sfrac{3R}{4}}\times\Rd} (\abs{\nabla_v g}^2 + \abs{\h}^2 &+ p^2)\Gamma(r+1,x-y)\,dm  \\
&\leq \delta\int_{Q_{\sfrac{3R}{4}}\times\Rd} g^2 \Gamma(r+1+C,x-y)\,dm 
+ \exp(-cR)\norm{g}^2_{\uL^2(Q_R;L^2_\gamma)}\,.
\end{align*}
Taking $\delta$ sufficiently small and plugging these bounds into \eref{vhkpoincare}, we now have that
\begin{align*}
\int_{Q_{\sfrac{3R}{4}}\times\Rd} g^2 \Gamma(r,x-y)\,dm &\leq \frac{1}{2}\int_{Q_{\sfrac{3R}{4}}\times\Rd} g^2 \Gamma(r+C,x-y)\,dm \\
&\quad+ \inp*{\frac{Cr}{R}}^{2(m+1)} \norm{f}_{\uL^2(Q_R;L^2_\gamma)}^2 +  \exp(-cR)\norm{g}^2_{\uL^2(Q_R;L^2_\gamma)}\,.
\end{align*}
Taking the integral average over $y\in Q_r$ and using that $\abs{Q_{r+C}} \leq \tfrac{3}{2} \abs{Q_r}$ for $r\geq C$, we get, for $r\in [Cm, \tfrac{1}{8}R]$
\begin{align*}
\int_{Q_{\sfrac{3R}{4}}\times\Rd} g^2\inp*{\frac{\1_{Q_r}}{\abs{Q_r}} * \Gamma(r,\cdot)} \,dm 
&\leq \frac{3}{4} \int_{Q_{\sfrac{3R}{4}}\times\Rd}g^2\inp*{\frac{\1_{Q_{r+C}}}{\abs{Q_{r+C}}} * \Gamma(r+C,\cdot)} \,dm \\
&\quad+ \inp*{\frac{Cr}{R}}^{2(m+1)} \norm{f}_{\uL^2(Q_R)}^2 +  \exp(-cR)\norm{g}^2_{\uL^2(Q_R)}\,.
\end{align*}
Iterating in $r$ then gives that
\begin{equation*}
\int_{Q_{\sfrac{3R}{4}}\times\Rd} g^2\inp*{\frac{\1_{Q_r}}{\abs{Q_r}} *\Gamma(r,\cdot)} \,dm
\leq  \inp*{\frac{Cr}{R}}^{2(m+1)} \norm{f}_{\uL^2(Q_R; L^2_\gamma)}^2 +  \exp(-cR)\norm{g}^2_{\uL^2(Q_R; L^2_\gamma)}\,.
\end{equation*}
Thus, for $r\in [Cm, \tfrac{1}{8}R]$, we now have that
\begin{equation*}
\norm{g}^2_{\uL^2(Q_r; L^2_\gamma)} \leq \inp*{\frac{Cr}{R}}^{2(m+1)} \norm{f}_{\uL^2(Q_R; L^2_\gamma)}^2 +  \exp(-cR)\norm{g}^2_{\uL^2(Q_R; L^2_\gamma)}\,.
\end{equation*}
To bound the final term on the right hand side of the previous display, we can use \eref{lscorrectorbound} and that $R\geq Cm$ to see that$$\norm{g}_{\uL^2(Q_R; L^2_\gamma)} \leq \norm{f}_{\uL^2(Q_R;L^2_\gamma)} + \norm{\psi}_{\uL^2(Q_R; L^2_\gamma)} \leq C^m \norm{f}_{\uL^2(Q_R;L^2_\gamma)}\,.$$
From this, it follows that 
$$ \exp(-cR)\norm{g}_{\uL^2(Q_R;L^2_\gamma)} \leq C^m\exp(-cR)\norm{f}_{\uL^2(Q_R; L^2_\gamma)} \leq \inp*{\frac{Cr}{R}}^{m+1} \norm{f}_{\uL^2(Q_R; L^2_\gamma)}\,.$$
Thus, as $g = f - \psi$, we have now shown that 
\begin{equation}
\label{e.almostlargescale}
\norm{f-\psi}_{\uL^2(Q_r; L^2_\gamma)} \leq \inp*{\frac{Cr}{R}}^{m+1} \norm{f}_{\uL^2(Q_R; L^2_\gamma)}\,.
\end{equation}

The only remaining problem is that the $\psi$ chosen above is not necessarily a solution of $L$. To fix this issue, we show that we can replace $\psi$ by an element of $\A_m$ which is a solution to $L$ and maintain the bounds shown above. To this end, it only remains to estimate $p := L\psi$. 
Applying \eref{polylinftybound} from the proof of \lref{polylem}, with $f-\psi$, we see that
$$ \norm{p}_{L^\infty(Q_r)} \leq \frac{C^m}{r^2} \norm{f-\psi}_{\uL^2(Q_r; L^2_\gamma)} \leq \frac{1}{r^2} \inp*{\frac{Cr}{R}}^{m+1} \norm{f}_{\uL^2(Q_R; L^2_\gamma)}\,.$$
Similarly, for $k\in\{0,\ldots, m\}$ and $r\in [Cm, \tfrac{1}{16}R]$, we can apply this bound as well as \eref{almostlargescale} and \lref{fdbound} to compute that
\begin{align*}
\norm{D^k p}_{L^\infty(Q_r)} &\leq \frac{C^{m-k}}{r^2} \norm{D^k(f-\psi)}_{\uL^2(Q_r; L^2_\gamma)} \\
&\leq \frac{C^{m-k}}{r^2}\inp*{\frac{Cr}{R}}^{m+1-k} \norm{D^k f}_{\uL^2(Q_{\sfrac{R}{2}}; L^2_\gamma)} \\
&\leq \frac{C^{m-k}}{r^2}\inp*{\frac{Cr}{R}}^{m+1-k}\inp*{\frac{Ck}{R}}^{k}  \norm{f}_{\uL^2(Q_R; L^2_\gamma)} \\
&\leq \frac{1}{r^2} \inp*{\frac{k}{r}}^k \inp*{\frac{Cr}{R}}^{m+1}\norm{f}_{\uL^2(Q_R; L^2_\gamma)}\,.
\end{align*}
Using this bound, \lref{hetpolysolveseq}, and \eref{derivfdbound}, we can find $\xi\in \A_m$ with $L \xi  = p$ and 
$$\norm{\xi}_{\uL^2(Q_r; L^2_\gamma)} \leq C r^2 \sum_{k=0}^m \inp*{\frac{Cr}{k+1}}^k \abs{D^k \hat{p}(0)} \leq \inp*{\frac{Cr}{R}}^{m+1} \norm{f}_{\uL^2(Q_R; L^2_\gamma)}\,.$$
Redefining $\psi := \psi + \xi$ and plugging this bound into \eref{almostlargescale}, we now have that 
$$\norm{f-\psi}_{\uL^2(Q_r; L^2_\gamma)} \leq \inp*{\frac{Cr}{R}}^{m+1} \norm{f}_{\uL^2(Q_R; L^2_\gamma)}\,,$$
where $\psi$ is a solution of $L$.
\end{proof}

\subsection*{Acknowledgments} 
The author was partially supported by NSF Grant DMS-2000200. 

\small
\bibliographystyle{abbrv}
\bibliography{hypoelliptic}
\end{document}